\documentclass[11pt]{amsart}

\usepackage[table]{xcolor}

\usepackage{mathscinet}
\usepackage{ytableau,tikz,varwidth}
\usepackage[all,cmtip]{xy}
\usetikzlibrary{calc, cd}
\usepackage[enableskew]{youngtab}
\usepackage[top=1in, right=1in, left=1in, bottom=1in]{geometry}
\usepackage{enumerate, amssymb, array}
\usepackage[normalem]{ulem}
\usepackage{scalerel}
\usepackage{longtable}
\usepackage{hyperref}

\newtheorem{thm}{Theorem}[section]
\newtheorem{prop}[thm]{Proposition}

\newtheorem{cor}[thm]{Corollary}
\newtheorem{conj}[thm]{Conjecture}

\theoremstyle{definition}

\newtheorem{defn}[thm]{Definition}

\newcommand{\sbu}{{\raisebox{1pt}{\scaleto{\bullet}{2.5pt}}}}
\newcommand{\bu}{\sbu}

\newcommand{\PP}{\mathbb{P}}
\newcommand{\ZZ}{\mathbb{Z}}

\newcommand{\cH}{\mathcal{H}}

\newcommand{\cO}{\mathcal{O}}
\newcommand{\cP}{\mathcal{P}}
\newcommand{\cQ}{\mathcal{Q}}

\newcommand{\cV}{\mathcal{V}}

\newcommand{\comment}[1]{}
\newcommand{\on}{\operatorname}
\newcommand{\Spec}{\operatorname{Spec}}
\newcommand{\Gr}{\mathrm{Gr}}

\newcommand{\cW}{\mathcal{W}}

\newcommand{\inv}{\operatorname{inv}}

\newcommand{\Fr}{\mathrm{Fr}}
\newcommand{\cpb}{\mathcal{P}^{\sbu}}
\newcommand{\cpbs}{\cpb_1, \cdots, \cpb_\ell}
\newcommand{\cqb}{\mathcal{Q}^{\sbu}}

\newcommand{\cvb}{\mathcal{V}^{\sbu}}
\newcommand{\cub}{\mathcal{U}^{\sbu}}
\newcommand{\cwb}{\mathcal{W}^{\sbu}}
\newcommand{\pb}{P^\sbu}
\newcommand{\qb}{Q^\sbu}
\newcommand{\vb}{V^\sbu}
\newcommand{\ab}{A^\sbu}
\newcommand{\bb}{B^\sbu}
\newcommand{\GL}{\mathrm{GL}}
\newcommand{\sigs}{\sigma_1,\cdots,\sigma_\ell}

\newcommand{\Grdab}{G^{r,\alpha,\beta}_d}
\newcommand{\Rab}{R_{A^\sbu,B^\sbu}}
\newcommand{\Rst}{R_{\sigma,\tau}}
\newcommand{\Ds}{D_{\sigma}}
\newcommand{\Dst}{D_{\sigma,\tau}}

\newcommand{\Rabj}{R_{A_j^\sbu,B_j^\sbu}}

\newcommand{\hide}[1]{}

\DeclareMathOperator{\Pic}{Pic}

\DeclareMathOperator{\Fl}{Fl}
\DeclareMathOperator{\End}{End}
\DeclareMathOperator{\Fix}{Fix}

\DeclareMathOperator{\Ess}{Ess}

\newtheorem{Definition}[thm]{Definition}
\newenvironment{definition}
  {\begin{Definition}\rm}{\end{Definition}}

\newtheorem{Example}[thm]{Example}
\newenvironment{example}
  {\begin{Example}\rm}{\end{Example}}

\newtheorem{Fact}[thm]{Fact}
\newenvironment{fact}
  {\begin{Fact}\rm}{\end{Fact}}

\newtheorem{Theorem}[thm]{Theorem}
\newenvironment{theorem}
  {\begin{Theorem}\rm}{\end{Theorem}}
  
\newtheorem{Lemma}[thm]{Lemma}
\newenvironment{lemma}
  {\begin{Lemma}\rm}{\end{Lemma}}

\newtheorem{Remark}[thm]{Remark}
\newenvironment{remark}
  {\begin{Remark}\rm}{\end{Remark}}

\newtheorem{Proposition}[thm]{Proposition}
\newenvironment{proposition}
  {\begin{Proposition}\rm}{\end{Proposition}}

\newtheorem{Corollary}[thm]{Corollary}
\newenvironment{corollary}
  {\begin{Corollary}\rm}{\end{Corollary}}

\theoremstyle{remark}

\newcommand \exnow[1]{\begin{example}{#1}\end{example}}
\newcommand \lemnow[1]{\begin{lemma}{#1}\end{lemma}}

\newcommand \enumnow[1]{\begin{enumerate}{#1}\end{enumerate}}

\AtBeginDocument{%
  \def\MR#1{}
}

\title{Relative Richardson varieties}
\author[M. Chan]{Melody Chan}\address{Department of Mathematics, Brown University, Box
1917, Providence, RI 02912}\email{melody\_chan@brown.edu}

\author[N. Pflueger]{Nathan Pflueger}\address{Department of Mathematics and Statistics, Amherst College, Amherst, MA 01002}\email{npflueger@amherst.edu}
\date{\today}

\begin{document}

\begin{abstract}
  A Richardson variety in a flag variety is an intersection of two Schubert varieties defined by transverse flags.  We define and study {\em relative Richardson varieties}, which are defined over a base scheme with a vector bundle and two flags. To do so, we generalize transversality of flags to a relative notion, versality, that allows the flags to be non-transverse over some fibers. Relative Richardson varieties share many of the geometric properties of Richardson varieties. We generalize several geometric and cohomological facts about Richardson varieties to relative Richardson varieties. We also prove that the local geometry of a relative Richardson variety is governed, in a precise sense, by the two intersecting Schubert varieties, giving a generalization, in the flag variety case, of a theorem of Knutson-Woo-Yong; we also generalize this result to intersections of arbitrarily many relative Schubert varieties.  We give an application to Brill-Noether varieties on elliptic curves, and a conjectural generalization to higher genus curves.
 \end{abstract}

\maketitle

\section{Introduction} \label{sec:intro}

A {\em Richardson variety} is an intersection of two Schubert varieties defined with respect to transverse flags in a vector space. Here we are concerned with Schubert varieties in Grassmannians and flag varieties.  A very simple example is the subvariety of the Grassmannian $\mathbb{G}(1,3)$ parametrizing lines in $\PP^3$ that meet two fixed skew lines.  

Richardson varieties are well-known to be rational, normal, and Cohen-Macaulay, and to have rational singularities, hence they have Euler characteristic 1.  Moreover, the singularities of a Richardson variety are governed entirely by the singularities of the two Schubert varieties: Knutson, Woo and Yong show that the singular points of a Richardson variety are exactly the points that are singular in either one of the Schubert varieties \cite{knutson-woo-yong-singularities}. 

In this article we generalize all of these basic results to a relative context. Throughout this paper, we fix an algebraically closed field $k$ of any characteristic. All schemes are assumed to be finite-type over $k$, and by a ``point'' of a scheme we will always mean a closed point.
We consider $\ell$ flag bundles within a vector bundle on a base scheme $S$, and we allow the flag bundles to become nontransverse over some points of $S$ in a controlled manner: subject to the condition of \emph{versality}.

Given a rank-$d$ vector bundle $\cH$ over a base scheme $S$, write $\Fr(\cH)$ for the frame bundle of $\cH$ and $\Fl(d)$ for the variety of complete flags in $k^d$. Then we define complete flag bundles $\cpb_1,\ldots,\cpb_\ell$  to be {\em versal} if the induced map $\Fr(\cH)\to\Fl(d)^\ell$ is a smooth morphism.  Versality usefully generalizes transversality to a relative context such that properties enjoyed locally by transverse intersections are still enjoyed by versal intersections. But versality is more general than transversality in every fiber.  

A simple example is a 1-parameter family of two complete flags in $\PP^3$ which are transverse except over a reduced point $p$, where the two 2-dimensional subspaces (lines $L_1, L_2$ in $\PP^3$) meet at a point rather than being skew. At each point of this family, one may consider the parameter space of lines in $\PP^3$ meeting both $L_1$ and $L_2$. These parameter spaces form what we will call a \emph{relative Richardson variety}, which is in this case a family of smooth quadric surfaces (parameterizing lines through two fixed skew lines) degenerating to a transverse pair of planes over the special point $p$ (one plane parameterizes lines through the intersection point of the fixed lines, while the other parameterizes lines coplanar with the two fixed lines). We now describe the general situation.

Suppose $\cH$ is a rank $d$ vector bundle over a base scheme $S$, and $\cpb$ is a complete flag of subbundles.  Fix a nest of sets
\begin{equation}\label{eq:sample-sequence}
\ab = (\{0,\ldots,d\!-\!1\} = A^{i_0} \supset A^{i_1} \supset \cdots \supset A^{i_s} = \emptyset)
\end{equation}
in which $|A^{i_j}| = d-i_j$.
Let $\pi\colon\Fl(i_0,\ldots,i_s;\cH)\to S$ denote the relative partial flag variety, equipped with tautological flag bundle $\cvb$ inside the rank $d$ vector bundle $\pi^* \cH.$
Define the $S$-scheme $X_{\ab}(\cpb)$ to be the subscheme 
$$\{x\in \Fl(i_0,\ldots,i_s;\cH): \dim (\cV^{i_j})_x \cap (\pi^* \cP^a)_x \ge \#\{a'\in A^{i_j} : a'\ge a\} \text{ for all }j,a \},$$
with scheme structure from its description as a degeneracy locus in the usual way, as recalled in \S\ref{ssec:degen}.
The first main theorem is as follows. The proof appears at the end of Section~\ref{sec:rr}.
\begin{thm} \label{thm:relativeRichardson}
Let $S$ be a smooth irreducible $k$-scheme, let $\cH$ be a rank-$d$ vector bundle on $S$, and let $\cP^\bu,\cQ^\bu$ be a versal pair of complete flags in $\cH$. 
Let $\ab, \bb$ be nests of sets as in Equation \eqref{eq:sample-sequence}.

\begin{enumerate}
\item The relative Richardson variety $$\Rab = X_{\ab}(\cpb) \cap X_{\bb}(\cqb)$$ is normal and Cohen-Macaulay of pure codimension $\inv(\omega \sigma) + \inv(\omega \tau)$ in the partial flag variety $\Fl(i_0,\ldots,i_s;\cH)$.
Here $\inv$ is the inversion number, $\omega$ is the descending permutation, and the permutations $\sigma=\sigma(\ab)$, $\tau =\sigma(\bb)$ are the decreasing completions of $\ab$ and $\bb$, as defined in \S\ref{ssec:perms}.
\item Letting $S'$ denote the scheme-theoretic image of $\Rab$ in $S$, the morphism $\pi\colon\Rab\rightarrow S'$ satisfies
\begin{eqnarray*}
\pi_\ast \cO_{\Rab} &=& \cO_{S'}, \textrm{ and}\\
R^i \pi_\ast \cO_{\Rab} &=& 0 \textrm{ for all $i > 0$}.
\end{eqnarray*}
\end{enumerate}
\end{thm}
\noindent Moreover, the scheme-theoretic image $S'$ is exactly understood.  It is the subscheme $D_{\tau \star \sigma^{-1}} (\cpb;\cqb)$ of $S$ where $\cpb$ and $\cqb$ meet with permutation bounded above, in Bruhat order, by the Demazure product $\tau \star \sigma^{-1}$. See Theorem~\ref{thm:rabgood} for details, and Fact \ref{fact:demazure} for recollections on the Demazure product.

\begin{cor} \label{cor:rrchi}
With the hypotheses of Theorem \ref{thm:relativeRichardson},
$$H^i(\Rab,\cO_{\Rab}) \cong H^i(S', \cO_{S'}),$$
and in particular
$$\chi( \Rab, \cO_{\Rab}) = \chi(S', \cO_{S'}).$$
\end{cor}

\begin{remark}
When the base $S$ is $\Spec k$, a versal pair of flags is the same as a transverse pair of flags in a fixed vector space, and Theorem \ref{thm:relativeRichardson} directly generalizes several facts about the geometry and cohomology of Richardson varieties. Corollary \ref{cor:rrchi} generalizes the fact that Richardson varieties have algebraic Euler characteristic $1$.
\end{remark}

\begin{remark}
We will say that the map $\Rab \to S'$ is a \emph{cohomological equivalence}. Our analysis of the cohomological properties of this map is analogous in several ways to the results in Section $1$ of \cite{anderson-chen-tarasca-k-classes}.
Indeed, the variety $\Omega_{\mathbf{p,q}}$ defined in \cite{anderson-chen-tarasca-k-classes} is a special case of a relative Richardson variety, namely in the case where the flag variety is a Grassmannian, and $W_{\mathbf{p,q}}$ is its image. The hypotheses in \cite{anderson-chen-tarasca-k-classes} are weaker; versality is not required. Under their hypotheses, they prove a K-theoretic equivalence statement, weaker than cohomological equivalence but stronger than equality of Euler characteristic. This allows \cite{anderson-chen-tarasca-k-classes} to give an independent proof of the main result of \cite{chan-pflueger-euler}, which we prove there using the cohomological equivalence results of this paper.
\end{remark}

We can also describe the smooth locus of $\Rab$, as
\begin{equation}\label{eq:intro-sm}
(\Rab)^\mathrm{sm} = (X_{\ab})^\mathrm{sm} \cap (X_{\bb})^\mathrm{sm}.
\end{equation}
In fact, our second main theorem generalizes this to intersections of arbitrarily many relative Schubert varieties and proves a much stronger result about the singularities of such an intersection.  The most general statement is in Theorem~\ref{thm:kwy}, and applies to $\ell$-fold intersections of degeneracy loci defined with respect to versal flags.  Applied to relative Schubert varieties, we obtain the following special case of Theorem~\ref{thm:kwy}. The proof appears at the end of Section \ref{sec:KWY}.

\begin{thm} \label{thm:kwyIntro}
Let $P$ be an \'etale-local property of finite-type $k$-schemes that is preserved by products with affine space. Suppose that there is an integer $\ell$ and a function $f_{P,\ell}$ such that for any finite-type $k$-schemes $X_1, \cdots, X_\ell$ and point $x \in \prod X_i$,
$$P\left( x, \prod X_i \right) = f_{P,\ell}\left( P(\pi_1(x),X_1), \cdots P(\pi_\ell(x), X_\ell) \right).$$

Let $\cpbs$ be versal complete flags on a scheme $S$, and $\ab_1,\cdots,\ab_\ell$ be nests of sets, each with the same coranks. Then for every point $x \in X_{\ab_1}(\cpb_1) \cap \cdots \cap X_{\ab_\ell}(\cpb_\ell)$,
$$P(x, X_{\ab_1}(\cpb_1) \cap \cdots \cap X_{\ab_\ell}(\cpb_\ell)) = f_{P,\ell} \Big( P(x, X_{\ab_1}(\cpb_\ell)), \cdots, P(x, X_{\ab_\ell}(\cpb_\ell)) \Big).$$
\end{thm}

This theorem generalizes the flag variety case of a theorem of Knutson, Woo, and Yong \cite{knutson-woo-yong-singularities}, both to the relative setting and to $\ell \geq 3$. Note that the results of \cite{knutson-woo-yong-singularities} apply to general Schubert varieties, whereas our results are specific to Schubert varieties in flag varieties.

\begin{remark}
We show in Example \ref{ex:transverseVersal} that a triple of fixed flags is never versal except in trivial cases. Therefore Theorem \ref{thm:kwyIntro} does not apply to intersections of three Schubert varieties in a fixed flag variety; the generalization to $\ell \geq 3$ depends in an essential way on the relative context.
\end{remark}

\subsection*{Application to Brill-Noether varieties}  Relative Richardson varieties arise naturally in the study of Brill-Noether varieties in \cite{chan-pflueger-euler}. Let $E$ be an elliptic curve; we will realize twice-pointed Brill-Noether varieties $\Grdab(E,p,q) \rightarrow S=\Pic^d(E)$ as relative Richardson varieties.  
\begin{corollary}
The schemes $\Grdab(E,p,q)$ are relative Richardson varieties.
\end{corollary}
\noindent (See Corollary 6.2.)  These varieties are the main building block in the proof in \cite{chan-pflueger-euler} of an Euler characteristic formula for Brill-Noether varieties, which uses limit linear series and degenerations of genus $g$ curves to chains of elliptic curves.   Corollary~\ref{cor:rrchi} is used to deduce the Euler characteristics of these building blocks.

In addition, by recognizing twice-pointed Brill-Noether varieties as relative Richardson varieties, we obtain another application: Equation~\ref{eq:intro-sm} gives a new proof of the main result of \cite{chan-osserman-pflueger-gieseker} in the case of elliptic curves, characterizing the singular locus of $\Grdab(E,p,q)$.  Previously, the description of the singular locus of twice-pointed Brill-Noether varieties for elliptic curves was obtained in \cite{chan-osserman-pflueger-gieseker} using an explicit analysis of vertical and horizontal tangent spaces at points in $\Grdab(E,p,q)$ relative to the map to the Picard variety $\Pic^d(E).$

We conjecture that Brill-Noether varieties of a twice-marked curve $C$ of higher genus are also isomorphic as $\Pic^d(C)$-schemes to relative Richardson varieties; we discuss this conjecture and its consequences in Section \ref{sec:brillnoether}.

\bigskip

\noindent {\bf Acknowledgments.}   We are grateful to Dave Anderson, Allen Knutson, Alex Woo, and Alex Yong for their correspondence regarding an earlier version of this manuscript, in particular for explaining the connection to the Demazure product. We also thank Jonathan Wise for helpful conversations leading to Remark \ref{rem:stacky}. MC was supported by NSF DMS-1701924,  NSF CAREER DMS-1844768, and a 2018 Sloan Research Fellowship.

\section*{Notation}

We collect for convenience the notation used throughout this paper. Where applicable, page numbers where the notation is first discussed are provided.

\bigskip

\begin{longtable}{rlr}
$S,T,F$ & Finite-type $k$-schemes\\
$\cH$ & A vector bundle\\
$d$ & The rank of $\cH$\\
$[d]$ & The set $\{0,1,\cdots, d-1\}$\\
$\cpb, \cqb, \cvb$ & Flags in $\cH$\\
$\cV^a$ & The stratum of codimension $a$ in $\cvb$\\
$M(H; P_1^\bu, \cdots, P_\ell^\bu)$ & Space of relative first-order deformations of $\ell$ flags & p. \pageref{loc:mdefn}\\
$\delta_x(\cH; \cP^\bu_1, \cdots, \cP^\bu_\ell)$ & Induced map $T_x S \rightarrow M$ & p. \pageref{loc:deltadef}\\
$A^\bu$, $B^\bu$ & Nests of sets & p. \pageref{loc:nest}\\
$\sigma,\tau$ & Permutations of $[d]$\\
$\sigma(A^\bu)$ & Decreasing completion of a nest of sets & p. \pageref{loc:dc}\\
$r^{\sigma}(a,b)$ & Rank function of a permutation & p. \pageref{loc:rs}\\
$\tau \star \sigma$ & Demazure product of permutations & p. \pageref{fact:demazure}\\
$\Fl(\cH)$ & Complete (relative) flag variety of a vector bundle $\cH$\\
$\Fl(i_0, i_1, \cdots, i_s; \cH)$ & Flag variety of $\cH$ with strata of codimensions $\{i_j\}$ \\
$\Fl(d)$ & Variety of complete flags in $k^d$ \\
$D_{\sigma}(\cvb; \cpb)$ & Degeneracy locus where $\cvb$ meets $\cpb$ as prescribed by $\sigma$ & p. \pageref{def:dst}\\
$D_{\sigs}(\cvb; \cpbs)$ & Intersection of $\ell$ degeneracy loci & p. \pageref{loc:dss}\\
$X_{\ab}(\cpb)$& (Relative) Schubert variety defined by a nest of sets & p. \pageref{loc:xa}\\
$X_\sigma(\cpb)$ & (Relative) Schubert variety defined by a permutation & p. \pageref{loc:xs}\\
$R_{\ab,\bb}(\cpb,\cqb)$ & Relative Richardson variety defined by two nests of sets & p. \pageref{loc:rab}\\
$R_{\sigma,\tau}(\cpb,\cqb)$ & Relative Richardson variety defined by two permutations& p. \pageref{loc:rst}\\
\end{longtable}

\bigskip

In all of these notations, we often omit the arguments (e.g. write simply $\Rst$) where they are clear from context. A tilde over a symbol (e.g., $\widetilde{D}_\sigma$) indicates the open locus where the defining inequalities of the object in question hold with equality.
\section{Preliminaries}\label{sec:prelim}

This section summarizes background material needed for this paper. General references for this material include the expository article \cite{brion-lectures} and book \cite{fulton-young} for the Schubert varieties in flag varieties and the Bruhat order, \cite{fulton-pragacz} for a nice explanation, with many pictures, of rank functions and their relationship to permutations, and \cite{fulton-flags} for degeneracy loci of flags of vector bundles, as well as the essential set of a permutation. We also refer the reader to \cite[\S 5.6]{demazure-desingularisation} for the Demazure product, as it relates to Schubert varieties. In several cases we use different notation conventions than these sources, more natural to our application; we explain these choices in this section.

\subsection{Permutations and nested sequences}\label{ssec:perms}  We begin with combinatorial conventions. We write $[d] = \{0,\ldots,d-1\}$ and write $S_d$ for the permutation group of $[d]$. Given a permutation $\sigma\in S_d$, we will write $\sigma = (\sigma_0,\ldots,\sigma_{d-1})$ in one-line notation, i.e.~$\sigma_i = \sigma(i)$.  The {\em inversion number} of 
 a permutation $\sigma\in S_d$ is
 $$\on{inv}(\sigma)=\# \{(i,j)\in [d]^2~|~ i<j\text{ and } \sigma_i > \sigma_j\}.$$
We will denote by $\omega$ the descending permutation $\omega(i) = d-1-i$. Observe that for all $\sigma \in S_d$, $\inv(\omega \sigma) = \binom{d}{2} - \inv(\sigma)$, the number of ``non-inversions'' of $\sigma$.

A \emph{nest of sets} \label{loc:nest} is a sequence
$$A^\bu = ([d] = A^{i_0}  \supset A^{i_1} \supset \cdots \supset A^{i_s} = \emptyset),$$
where $|A_{i_j}| = d - i_j$. The numbers $i_s$ are called the \emph{coranks} of $\ab$. Note that we require $i_0 = 0$ and $i_s = d$ for convenience later.

Define the {\em decreasing completion} \label{loc:dc} $\sigma(A^\bu) \in S_d$ of $A^\bu$ to be the permutation obtained by writing the elements of $A^{i_0} \setminus A^{i_1}$ in decreasing order, then the elements of $A^{i_1}\setminus A^{i_2}$ in decreasing order, and so on.  
For example, the decreasing completion of
$$\{0,1,2,3,4\}  \supset \{0,1,3\} \supset \emptyset$$
is $(4,2,3,1,0)$. 
For a nest of sets $\ab$, we define $\on{inv}(\ab) = \on{inv}(\sigma(\ab)).$
Decreasing completion provides a bijection between $S_d$ and nests of sets of coranks $(0,1,\cdots,d-1)$. Such a nest $\ab$ is called \emph{complete}. We will often identify complete nests of sets with permutations.

\subsection{Flags}

Let $H$ be a $d$-dimensional vector space over $k$.  Write 
$$P^\bu = (P^0 \supset P^1 \supset \cdots\supset P^d = 0)$$
for a complete flag of subspaces of $H$, where $P^i$ has codimension $i$.  The relative position of two flags $\pb$ and $\qb$  uniquely defines a \emph{rank function}
$$r(a,b) = \dim P^a \cap Q^b.$$

The rank function of two complete flags can be encoded by a permutation. For any $\sigma \in S_d$, define the rank function of $\sigma$ by
$$\label{loc:rs}
r^\sigma(a,b) = \# \{a' \in [d]: a' \geq a \textrm{ and }  \sigma(a') \geq b \}.$$
Note that this notation does not exclude the cases $a \geq d$ or $b \geq d$, where we define $r^\sigma(a,b) = 0$.

\begin{fact} \label{fact:bruhat} For any two complete flags $\pb, \qb$, there exists a unique $\sigma \in S_d$ such that $\dim P^a \cap Q^b = r^\sigma(a,b)$ for all $a,b \in [d]$, called the \emph{permutation associated to $\pb, \qb$}. The following are equivalent.
\begin{enumerate}
\item The permutation associated to $\pb,\qb$ is $\sigma$.
\item There exists a basis $v_0, v_1, \cdots, v_{d-1}$ of $H$ such that $\{v_a, \cdots, v_{d-1}\}$ is a basis for $P^a$ for all $a$, and $\{v_{\sigma(b)}, \cdots, v_{\sigma(d-1)}\}$ is a basis for $Q^b$ for all $b$.
\end{enumerate} 
\end{fact}
\noindent For example, $\pb=\qb$ if and only if $\sigma = \mathrm{id}$.  At the other extreme, flags $P^\bu,Q^\bu$ are {\em transverse} if their associated permutation is $\omega$. Explicitly, $\pb,\qb$ are transverse if and only if 
$$\dim P^a \cap Q^b = \max(d - a - b, 0)$$
for all $i, j$; that is, every pair of subspaces meets transversely.
Call $\pb,\qb$ {\em almost-transverse} if their associated permutation differs from $\omega$ by an adjacent transposition, or equivalently $\inv(\omega \sigma) = 1$.

A rank function $r^\sigma$ is uniquely determined by its values on a fairly small subset of its domain.

\begin{defn}
The \emph{essential set} of a permutation $\sigma \in S_d$ is
$$\Ess(\sigma) = \left\{ (a,b):\ 1 \leq a,b < d,\ \sigma(a-1) < b \leq \sigma(a) \mbox{ and } \sigma^{-1}(b-1) < a \leq \sigma^{-1}(b)\right\}.$$
\end{defn}


The essential set was introduced in \cite{fulton-flags}, although we define it slightly differently here; see Remark \ref{rem:essDiff}.
The importance of the essential set is reviewed in Fact \ref{fact:essSet}. 

The set $S_d$ has a partial order, the {\em Bruhat order}: $\sigma \le \tau$ in Bruhat order if and only if $r^\sigma(a,b) \geq r^\tau(a,b)$ for all $a$ and $b$.  See, e.g., \cite[\S10.5]{fulton-young}. By semicontinuity, the associated permutation of two varying flags is lower semi-continuous in the Bruhat order.

Let $\Fix P^\bu$ denote the vector subspace of $\End H$ consisting of $\phi: H \rightarrow H$ such that $\phi(P^i) \subseteq P^i$ for all $i$.   
The following characterization of $\inv(\sigma)$, for a permutation associated to flags $\pb,\qb$, will be convenient later.

\begin{fact} \label{fact:invFix}
If $\sigma$ is the permutation associated to flags $P^\bu, Q^\bu$, then 
$$\inv(\omega \sigma) = \dim \End H - \dim \left( \Fix P^\bu + \Fix Q^\bu \right).$$
In particular, $P^\bu$ and $Q^\bu$ are transverse if and only if $\Fix P^\bu + \Fix Q^\bu = \End H$.
\end{fact}

Fact \ref{fact:invFix} can be proved using a straightforward argument characterizing $\Fix P^\bu \cap \Fix Q^\bu$ in terms of a basis of the type described in Fact \ref{fact:bruhat}.


\subsection{The Demazure product}

To state our main results, we require an associative operation $\star$ on $S_d$ called the Demazure product.

\begin{fact} \label{fact:demazure}
For any two permutation $\sigma, \tau \in S_d$, there exists a unique permutation $\tau \star \sigma$ such that
\begin{equation}\label{eq:r-star}
r^{\tau \star \sigma}(a,b) = \max_{0 \leq k \leq d} \left( r^{\sigma}(a,k) + r^{\tau}(k,b) - (d-k) \right).
\end{equation}
The operation $\star$ defined in this way is associative, and satisfies $(\sigma \star \tau)^{-1} = \tau^{-1} \star \sigma^{-1}$. When one of the permutations is a simple transposition $s$ (a transposition of two adjacent elements of $[d]$), then 
\begin{equation}\label{eq:star-recursion}
\tau \star s = \begin{cases} \tau & \mathrm{ if } \inv(\tau s) < \inv(\tau) \\ \tau s & \mathrm{if \inv(\tau s) > \inv(\tau)} \end{cases}.   
\end{equation}
\end{fact}

Equation \ref{eq:r-star} is motivated by the following observation: if $P^\bu, Q^\bu, R^\bu$ are three flags, $\sigma$ is the permutation associated to $P^\bu, Q^\bu$, and $\tau$ is the permutation associated to $Q^\bu, R^\bu$, then for all $a,b,k \in [d]$,
\begin{equation}
\dim P^a \cap R^b \geq  \dim P^a \cap Q^k + \dim Q^k \cap R^b - \dim Q^k = r^\sigma(a,k) + r^\tau(k,b) - (d-k).
\end{equation}
Therefore $\tau \star \sigma$ gives an upper bound on the permutation associated to $P^\bu, R^\bu$. In fact, one can deduce from Theorem \ref{thm:relativeRichardson} that $\tau \star \sigma$ is the minimal such permutation. 

The Demazure product was introduced and studied in \cite{demazure-desingularisation} and \cite{bernstein-gelfand-gelfand-schubert} for arbitrary Weyl groups. We briefly sketch a proof of Fact \ref{fact:demazure} for the benefit of the reader unfamiliar with these topics, as follows. One may use the right side of Equation \ref{eq:r-star} to define an operation on functions $[d] \times [d] \rightarrow [d]$. One can verify that this operation is associative by writing a composition of two products as a maximum taken over two variables, and Equation \ref{eq:star-recursion} may be verified, on the level of rank functions, by some casework. Finally, the existence of a permutation with the desired rank function may be proved by induction on $\inv(\sigma)$ by factoring $\sigma$ into simple transpositions. 


\subsection{Schubert varieties in flag varieties} \label{ssec:schubert}

Fix a vector space $H$ of dimension $d$, and let $F^\bu \in \Fl(H)$ be a fixed complete flag. Given $\sigma\in S_d$, define the \emph{Schubert variety} $X_\sigma$ by
$$\label{loc:xs}
X_\sigma = X_\sigma(F^\bu) = \{\vb \in \Fl(H)\colon \dim V^a\cap F^{b} \ge r^\sigma(a,b) \mbox{ for all } a,b \in [d]\}.$$

We write $\widetilde{X}_\sigma$ for the open locus where all these defining inequalities hold with equality.

\begin{remark}
Our conventions differ from those used in \cite{fulton-young} and elsewhere, since we index our flags by codimension, rather than dimension. We choose this convention because it is most natural for our application in \cite{chan-pflueger-euler}, where we stratify sections of a line bundle by their vanishing order at a point.
For example, in \cite{fulton-young}, the associated permutation $w$ of two flags $V_\bu,W_\bu$ (indexed by dimension) is defined by $\dim V_a \cap W_b = r_w(a,b)$, where $w$ is a permutation of $\{1,2,\cdots,d\}$ and the rank function is $r_w(a,b) = \# \{i \leq a: w(i) \leq b\}$. There are two ways to translate our notation to the notation of \cite{fulton-young}.

\begin{enumerate}
\item Define an isomorphism $i: \Fl(H) \rightarrow \Fl(H^\vee)$ by $i(\pb) = V_\bu$, where $V_a = (P^a)^\perp$. If $i(\pb) = V_\bu$ and $i(\qb) = W_\bu$, then 
$$\dim V_a \cap W_b = a+b-d + \dim P^a \cap Q^b,$$
from which it follows that $\dim P^a \cap Q^b = r^\sigma(a,b)$ if and only if $\dim V_a \cap W_b = r_w(a,b)$, where $w(i) = \sigma(i-1)+1$ (in one-line notation, $w$ is obtained by adding one to all entries of $\sigma$). So $i(X_\sigma)$ is equal to the Schubert variety denoted $X_w$ in \cite{fulton-young}.
\item Define $V_a = P^{d-a}$ and $W_b = Q^{d-b}$. If $\dim P^a \cap Q^b = r^\sigma(a,b)$ for all $a,b$, then $\dim V_a \cap W_b = r_w(a,b)$, where $w(i) = d- \sigma(d-i)$ (in one-line notation, $w$ is obtained by adding $1$ to all entries of $\omega \sigma \omega$). So our $X_\sigma$ is equal to the variety denoted $X_w$ in \cite{fulton-young}.
\end{enumerate}

\end{remark}

We collect facts about $X_\sigma$. It is well-known that $X_\sigma$ is irreducible, normal, and Cohen-Macaulay, of codimension $\inv(\omega \sigma)$ in $\Fl(H)$. A criterion for whether $X_\sigma$ is regular is given by Lakshmibai and Sandhya \cite{lakshmibai-sandhya-criterion}: $X_\sigma$ is regular if and only if $\sigma$ is a $3120$ and $2301$-avoiding permutation.\footnote{In standard notation, $4231$ and $3412$-avoiding.} 
The singular locus of $X_\sigma$ is closed and a union of Borel orbits; therefore, it must be a union of varieties $X_{\sigma'}$ for 
$\sigma'\le \sigma$.
Lakshmibai-Sandhya conjectured a combinatorial description of which Schubert subvarieties $X_{\sigma'}$ occur, and their conjecture was proven independently by several groups \cite{billey-warrington-maximal, cortez-singularities, kassel-lascoux-reutenauer-singular, manivel-lieu}. The description shows that $\sigma'$ ranges over all permutations that are derived from minimal $3120$ and $2301$ patterns in $\sigma$ by a certain combinatorial modification, see e.g., \cite[\S1]{manivel-lieu}.  We note that the singular locus of $X_{\sigma}$ is a union of $X_{\sigma'}$ for 
$\sigma'<\sigma$ 
ranging over a certain set of permutations having at least two fewer inversions than $\sigma$.  In particular, $X_{\sigma}$ is regular in codimension 1.

More generally, if $\ab$ is a nest of sets of coranks $0 = i_0 < \cdots < i_s = d$, then we define a Schubert variety in the partial flag variety $\Fl(i_0, \cdots, i_s; H)$ as follows.
$$\label{loc:xa}
X_{\ab}(F^\bu) = \{ \vb \in \Fl(i_0, \cdots, i_s; H):\ \dim V^{i_j} \cap F^b \geq r^{\sigma(\ab)}(i_j,b) \mbox{ for all } j \in [s], b \in [d] \}.
$$
We write $\widetilde{X}_{\ab}$ for the open locus where these defining inequalities hold with equality. The inverse image of $X_{\ab}$ under the forgetful map $\Fl(H) \rightarrow \Fl(i_0, \cdots, i_s; H)$ is equal to $X_{\sigma}(F^\bu)$. Since this forgetful map is a fiber bundle with smooth irreducible fibers, most of the geometric facts above carry over readily to $X_{\ab}$.


\subsection{Degeneracy loci and relative Schubert varieties}\label{ssec:degen}
We adopt the following notation convention. If $\cV,\cW$ are two sub-bundles of a vector bundle $\cH$ on a scheme $S$, we will write
$$\{x \in S:\ \dim \cV_x \cap \cW_x \geq r\}$$
as a shorthand for the subscheme defined by the degeneracy locus where the bundle map $\cV \rightarrow \cH/\cW$ has rank at most $\mathrm{rank}(\cV) - r$, defined locally as a determinantal variety in the usual way (e.g. as in \cite[\S 4]{fulton-flags} or \cite[\S II.4]{acgh}). In particular, we always mean this notation as a scheme-theoretic definition. 

We will be concerned with degeneracy loci of the following form. For $\cH$ a rank $d$ vector bundle on a scheme $S$, and $\cP^\bu, \cQ^\bu$ complete flags in $\cH$, we consider the subscheme
\begin{equation} \label{eq:prelimDs}
 \left\{x \in S:\ 
\dim ( \cP^a )_x \cap (\cQ^{b})_x 
\geq 
r^\sigma(a,b) \mbox{ for all } a,b \in [d]\right\}.
\end{equation}

In fact, many of the inequalities in this definition are redundant.

\begin{fact} \label{fact:essSet} (\cite[Lemma 3.10]{fulton-flags})
The scheme described by Equation \eqref{eq:prelimDs} is equal to the scheme
$$\left\{ x \in S:\ 
\dim ( \cP^a )_x \cap (\cQ^{b})_x 
\geq 
r^\sigma(a,b) \mbox{ for all } (a,b) \in \Ess(\sigma)\right\}.$$
\end{fact}

\begin{remark} \label{rem:essDiff}
The definition of the essential set in \cite{fulton-flags} is different from ours, because the degeneracy loci under consideration are defined by $\mathrm{rank}(E_p \rightarrow F_q) \leq r_w(q,p)$, where $E_1 \hookrightarrow E_2 \hookrightarrow \cdots E_n$ and $F_n \twoheadrightarrow F_{n-1} \twoheadrightarrow \cdots \twoheadrightarrow F_1$ are vector bundles indexed by rank. The essential set of \cite{fulton-flags} is the set of $(q,p)$ for which the condition $\mathrm{rank}(E_p \rightarrow F_q) \leq r_w(q,p)$ is essential. Our definition is obtained by a straightforward translation. 
\end{remark}

In light of Fact \ref{fact:essSet}, we make the following definition, which allows for partial flags.

\begin{defn} \label{def:dst}
Let $\cH$ a rank $d$ vector bundle on a scheme $S$, and $\cP^\bu = (\cP^{i_0} \supset \cdots \supset \cP^{i_s})$, $\cQ^\bu = (\cQ^{j_0} \supset \cdots \supset \cQ^{j_t} )$ be flags in $\cH$. Let $\sigma$ be any permutation such that $\Ess(\sigma) \subseteq \{i_0, \cdots, i_s\} \times \{j_0, \cdots, j_t\}$. Define a subscheme 
$$D_\sigma(\cP^\bu; \cQ^\bu) = \left\{x \in S:\ 
\dim ( \cP^a )_x \cap (\cQ^{b})_x 
\geq 
r^\sigma(a,b) \mbox{ for all } a,b \in \Ess(\sigma) \right\}.$$
When $\cpb,\cqb$ are complete flags, also define $\widetilde{D}_\sigma(\cP^\bu;\cQ^\bu)$ to be the open subscheme where we have equality $\dim (\cP^a)_x \cap (\cQ^b)_x = r^{\sigma}(a,b)$ for all $(a,b) \in [d]^2$ (not only those in $\Ess(\sigma)$). Thus $\widetilde{D}_\sigma(\cP^\bu; \cQ^\bu)$ is the locus where the two flags have associated permutation $\sigma$, and $D_\sigma(\cP^\bu; \cQ^\bu)$ is the locus where the two flags have associated permutation 
at most
$\sigma$ in Bruhat order.

\begin{remark} \label{rem:essentialDescents}
Suppose that $\cqb$ is complete, and the strata of $\cpb$ have coranks $0 = i_0 < \cdots < i_s = d$. Then the following is a useful sufficient condition for $D_\sigma(\cpb; \cqb)$ to be well-defined: for all $0 \leq j < s$, $\sigma(i_{j}) < \sigma(i_{j} + 1) < \cdots < \sigma(i_{j+1}-1)$. This condition ensures that $\sigma(a-1) > \sigma(a)$ for all $a$ except possibly when $a \in \{i_0, \cdots, i_s\}$ and thus $\Ess(\sigma) \subseteq \{i_0, \cdots, i_s\} \times [d]$. In other words, if $\ab$ is any nest of sets with coranks $i_0, \cdots, i_s$, then $D_{\sigma(\ab)}(\cpb; \cqb)$ is well-defined. 
\end{remark}

We mention some important geometric facts about degeneracy loci, stated at the level of generality we need; see \cite{fulton-flags} for the more general statement, including an intersection theory result.

\begin{fact} (\cite[Theorem 8.2]{fulton-flags}) \label{fact:fultonDegen}
Suppose $S$ is Cohen-Macaulay and pure dimensional. Then any component of $\Ds(\cpb; \cqb)$ has codimension at most $\inv(\omega \sigma)$. If the codimension of $\Ds$ is exactly $\inv(\omega \sigma)$, then $\Ds$ is Cohen-Macaulay.
\end{fact}

Given $\ell+1$ flags $\cvb, \cpb_1,\ldots,\cpb_\ell$ and $\ell$ permutations $\sigma_1,\ldots,\sigma_\ell$, we also use the following abbreviation.
$$
\label{loc:dss}
D_{\sigma_1,\ldots,\sigma_\ell}(\cvb; \cpb_1,\ldots,\cpb_\ell) = D_{\sigma_1}(\cvb; \cpb_1) \cap \cdots \cap D_{\sigma_\ell}(\cvb; \cpb_\ell).
$$
When the flags are clear from context, we will omit the arguments and write simply $D_\sigma$ or $D_{\sigma_1,\ldots,\sigma_\ell}$.
\end{defn}

\begin{remark} \label{rem:dsFromSchubert}
It is sometimes convenient to view $D_\sigma$ locally as the inverse image of a Schubert variety. If $U \subset S$ is an open subscheme on which $\cH$ is trivial, 
then we may 
choose completions of $\cpb$ and $\cqb$ 
and locally trivialize $\cH$ 
in a way that makes the completion of $\cqb$ constant. Then the completion of $\cpb$ defines a morphism $p: S \rightarrow \Fl(d)$ (under which the completion of $\cpb$ is the pullback of the tautological bundle), and we have, scheme-theoretically,
$$D_\sigma(\cpb; \cqb) = p^{-1} (X_\sigma).$$
\end{remark}

\begin{example} (Relative Schubert varieties)
Let $S$ be a scheme, $\cH$ a rank-$d$ vector bundle on $S$, and $\cP^\bu$ a complete flag in $\cH$. Let $\pi: \Fl(\cH) \rightarrow S$ denote the relative flag variety. For every $\sigma \in S_d$, there is a relative Schubert subvariety $X_\sigma(\cpb) \subseteq \Fl(\cH)$. These subvarieties are important special cases of the degeneracy loci defined above, namely
$$X_{\sigma}(\cpb) = D_\sigma(\cvb; \pi^\ast \cP^\bu)$$
where $\cvb$ is the tautological flag bundle. Similarly, we obtain relative Schubert varieties in partial flag varieties:
$$X_{\ab}(\cpb) = D_{\sigma(\ab)}(\cvb; \pi^\ast \cpb).$$
\end{example}

\section{Versality}

In this subsection, we define versality of complete flags, and prove several criteria for it. We work exclusively with complete flags in this subsection; results for incomplete flags can be deduced from the case of complete flags.

Let $\cH$ be a vector bundle of rank $d$ on a base scheme $S$.
Denote by $\Fl(d)=\Fl(k^d)$ the variety of complete flags in the standard vector space $k^d$, and denote by $\Fr(\cH) \rightarrow S$ the frame bundle of $\cH$. Then a complete flag $\cP^\bu$ in $\cH$ uniquely determines a morphism of schemes $\Fr(\cH) \rightarrow \Fl(d)$. 

\begin{definition} \label{def:versal}
With the notation above, suppose that $\cpb_1, \cdots, \cpb_\ell$ are complete flags in $\cH$, inducing a morphism
$$p:\ \Fr(\cH) \rightarrow \Fl(d)^\ell.$$
Call the $\ell$-tuple of flags $(\cP_1, \cdots, \cP_\ell)$ \emph{versal} if $p $ is a smooth morphism. Call the flags \emph{versal at $x \in S$} if they are versal when restricted to some neighborhood of $x$.
\end{definition}

\noindent Observe that any subset of a versal $\ell-$tuple of flags is again versal, since the projection $\Fl(d)^\ell \rightarrow \Fl(d)^{\ell'}$ is smooth for all $\ell' < \ell$.

When $\ell=1$, versality is automatic. When $\ell = 2$ and $S = \Spec k$, two flags are versal if and only if they are transverse, as is explained in Example \ref{ex:transverseVersal}. When $\ell = 2$ but $S$ is a more general scheme, Definition \ref{def:versal} is equivalent to a geometric condition (Lemma \ref{lem:versalityCriterion2}) that roughly says that the locus where $\cP^\bu, \cQ^\bu$ are nontransverse is stratified by smooth varieties of specific codimension. This stratification is indexed by permutations. 

Our first goal is a linear-algebraic criterion for versality (Proposition \ref{def:versalityCriterion}), for which we need some preliminary notions.
For complete flags $P^\bu_1, \cdots, P^\bu_\ell$ in a vector space $H$, let 
$$M=M(H; P^\bu_1, \cdots, P^\bu_\ell)=\mathrm{coker} (\End H \xrightarrow{\Delta} \prod_{i=1}^\ell \End H / \Fix P^\bu_i),
\label{loc:mdefn}
$$
where 
$\Delta$ is the diagonal map.

This vector space $M$ is the space of {\em relative} first-order deformations of the flags.  Indeed, each factor $\End H / \Fix P^\bu_i$ is naturally identified with the tangent space at $[P^\bu_i]$ to the flag variety, 
while the image of $\Delta$ corresponds to simultaneous deformations arising from a change of basis for $H$. 

Now given $x\in S$, let 
$$\delta_x = \delta_x(\cH; \cP^\bu_1, \cdots, \cP^\bu_\ell):\ T_x S \rightarrow M(\cH_x; (P^\bu_1)_x, \cdots, (P^\bu_\ell)_x) \label{loc:deltadef}$$
denote the natural linear map encoding the first-order deformations of the $P^\bu_i$ relative to each other induced by a first-order deformation in $S$. A precise definition is as follows.
Note the natural map $p\colon \Fr(\cH)\to \Fl(d)^\ell$ induced by the flags $\cP_1^\bu,\ldots, \cP_\ell^\bu$.  Shrink $S$ so that $\cH$ is trivial.  Any section $s\colon S\to \Fr(\cH)$ of $\Fr(\cH)\to S$ induces a linear map
\begin{equation} \label{eq:deltax}
\delta_x\colon T_x S \longrightarrow T_{(p\circ s)(x)} \Fl(d)^\ell \longrightarrow M,
\end{equation}
where the first arrow is the differential of $p\circ s$ at $x$, and the second is induced by $s(x)$.

Moreover, any two sections $s,s'$ are related by an element of $\mathrm{GL}_d(S)$. That is, there is a morphism $c\colon S\to \mathrm{GL}_d$ with the following commuting diagram.
$$\xymatrix{S\ar@{=}[d] \ar[r]^-{c\times s} &\mathrm{GL}_d \times \Fr(\cH) \ar[d] \ar[r] & \mathrm{GL}_d \times \Fl(d)^\ell \ar[d] 
\\ S \ar[r]^-{s'}          &\Fr(\cH) \ar[r] & \Fl(d)^\ell}
$$
Passing to tangent spaces verifies that $\delta_x$ does not depend on choice of $s$.


\begin{prop} \label{def:versalityCriterion}
Let $S$ be a scheme with a vector bundle $\cH$ and complete flags $\cP^\bu_1, \cdots, \cP^\bu_\ell$. For any $x \in S$, these flags are versal at $x$ if and only if $S$ is smooth at $x$ and $\delta_x$ is surjective.
\end{prop}

\begin{proof}
Fix $x \in S$ and a point $y \in \Fr(\cH)$ in the fiber over $x$. First, we claim that the map $\delta_x$ is surjective if and only if the differential $dp_y$ of the map $p\colon \Fr(\cH) \rightarrow \Fl(d)^\ell$ is surjective. Shrinking $S$ if necessary, choose a section $s: S \rightarrow \Fr(\cH)$ such that $y = s(x)$. We use the description of $\delta_x$ in Equation \ref{eq:deltax}. The kernel of the map $T_{(p \circ s)(x)} \Fl(d)^\ell \rightarrow M$ is equal to the tangent space to the $\GL_d$-orbit of $(p \circ s)(x)$. Since $p$ is equivariant, this is the image under $dp_{s(x)}$ of the $\GL_d$-orbit of $s(x)$. Therefore the map $T_{s(x)} \Fr(\cH) \rightarrow M$ is surjective if and only if $dp_{s(x)}$ is surjective. Furthermore, the $\GL_d$-orbit of $s(x)$ is the fiber of $x$ in $\Fr(\cH)$, so its tangent space is complementary to the image of $ds_x$. Therefore the image of $\delta_x$ is equal to the image of $T_{s(x)} \Fr(\cH) \rightarrow M$. Putting this together, $\delta_x$ is surjective if and only if $dp_{s(x)}$ is surjective.

Next, observe that since $\Fr(\cH)$ is a $\GL_d$-torsor, it is smooth at $y$ if and only if $S$ is smooth at $x$.

Suppose that the flags are versal at $x$. Since $\Fl(d)^\ell$ is a nonsingular variety, the structure map $\Fl(d)^\ell \rightarrow \Spec k$ is smooth, hence the composition $\Fr(\cH) \rightarrow \Spec k$ is smooth at $y$, i.e., $y$ is a smooth point of $\Fr(\cH)$ and $x$ is a smooth point of $S$. Since $p$ is a smooth morphism of nonsingular varieties in a neighborhood of $y$, the differential $dp_y$ is surjective \cite[10.4]{hartshorne-algebraic}. It follows that $\delta_x$ is surjective as well.

Now suppose $x \in S$ is a smooth point and $\delta_x$ is surjective. Then $y$ is a smooth point of $\Fr(\cH)$ and $dp_y$ is surjective, so $p$ is a morphism of nonsingular varieties with surjective differential around $y$. Hence $p$ smooth at $y$, and the flags are versal at $x$.
\end{proof}

\begin{remark} \label{rem:stacky}
The definition of versality and the criterion of Proposition \ref{def:versalityCriterion} has a stack-theoretic description, as follows. A choice of $\ell$ complete flags on $S$ is equivalent to a morphism $\overline{p}$ from $S$ to the quotient stack $\left[ \Fl(d)^\ell / \GL_d \right]$, which may be regarded as the moduli stack of $\ell$-tuples of flags. This morphism, along with the induced morphism $p: \Fr(\cH) \rightarrow \Fl(d)^\ell$ discussed above, form a cartesian diagram as follows.
$$
\xymatrix{
\Fr(\cH)\ar[r]^-{p} \ar[d] & \Fl(d)^\ell \ar[d] \\
S \ar[r]^-{\overline{p}} & \left[ \Fl(d)^\ell / \GL_d \right]
}
$$
The vertical arrows are $\GL_d$-torsors, and it follows that $p$ is smooth if and only if $\overline{p}$ is smooth. So the tuple of flags is versal if and only if it determines a smooth morphism to the moduli stack; this accords with the usual use of ``versal'' in deformation theory. 
The differential of $\overline{p}$ at $x \in S$ may be identified with a map from $T_x S$ to a two-term complex $\End H \xrightarrow{\Delta} \prod_{i=1}^\ell \End H / \Fix P^\bu_i$, which is surjective if and only if the linear map $\delta_x$ is surjective. Hence Proposition \ref{def:versalityCriterion} amounts to the fact that $\overline{p}$ is smooth if and only if it has smooth domain and surjective differential.
\end{remark}

\begin{example} \label{ex:transverseVersal}
(Versality of fixed flags)
A pair of two complete flags $\cP^\bu, \cQ^\bu$ are versal in a neighborhood of any smooth point $x \in S$ where $\cP^\bu_x$ and $\cQ^\bu_x$ are transverse. This is because $\Fix P^\bu_x + \Fix Q^\bu_x = \End H$ by Fact \ref{fact:invFix}, which implies (and is equivalent to) $M(H_x; P_x^\bu, Q_x^\bu) = 0$. 
If $S=\Spec k$, the flags are transverse if and only if they are versal. Therefore versality is a generalization of transversality.

When $\ell > 2$ and $d > 2$, then $\ell$ flags over $S=\Spec k$ are {\em never} versal since $$\dim \prod \End H/\Fix \pb_i = \ell \binom{d}{2} \geq 3 \binom{d}{2} \geq d^2 = \dim \End H,$$ so $\End H \rightarrow \prod \End H / \Fix P^\bu$ cannot be surjective because the kernel always contains the identity and is therefore nontrivial. When $d=2$, the only way versality over $\Spec k$ can occur is if $\ell \leq 3$ and the flags are distinct.
\end{example}

Adding a tautological flag bundle preserves versality:

\begin{lemma} \label{lem:versalityCriterion3}
With $\cH$ a vector bundle on a smooth scheme $S$, if $\cpb_1,\ldots,\cpb_\ell$ are versal complete flags on $S$, and $\pi:\Fl(\cH) \rightarrow S$ is the flag variety of $\cH$ with tautological bundle $\cV^\bu$, then $\pi^\ast \cpb_1, \ldots,\pi^\ast \cpb_\ell, \cV^\bu$ are versal flags of $\pi^\ast \cH$ on $\Fl(\cH)$. 
\end{lemma}

\begin{proof}
Fix $x\in S$, and define $P^\bu_i = (\cpb_i)_x$. Shrinking $S$, we may assume $\cH = H\times S$. For any point $y = [\qb] \in \Fl(H)$, versality at $(x,y) \in \Fl(\cH)$ is equivalent to the surjectivity of the linear map
$$T_{(x,y)} \Fl(\cH) \oplus \End H \to \prod \End H/\Fix \pb_i \times \End H/\Fix \qb.$$
We have $T_{(x,y)}\Fl(\cH) = T_x S \oplus T_y \Fl(H)$.  Then surjectivity of the above map follows from the fact that $T_x S \oplus \End H \to \prod \End H/\Fix \pb_i$ is surjective by the versality hypothesis, and $T_y \Fl(H) \to \End H /\Fix \qb$ is an isomorphism.
\end{proof}

 The linear maps $\delta_x$ provide a description of tangent spaces to degeneracy loci $D_\sigma(\cpb; \cqb)$ defined in Definition~\ref{def:dst}, and intersections thereof.

\begin{lemma} \label{lem:tspaceSSigma}Let $D_\sigma(\cpb;\cqb)$ be as in Definition~\ref{def:dst}.
\enumnow{
\item At any point $x \in D_\sigma( \cpb; \cqb)$,
$$T_x D_\sigma( \cpb; \cqb) \supseteq \ker \delta_x (\cH; \cpb, \cqb).$$
\item Moreover, if $x \in \widetilde{D}_\sigma(\cpb; \cqb)$, then equality holds.
}
\end{lemma}
\begin{proof}
Start by verifying (2) in the 
special case that $H=k^d$ and $S = \Fl(H)^2$ with $\cpb,\cqb$ the tautological flag bundles.  
For any $ x = (\pb,\qb) \in \widetilde{D}_\sigma( \cpb; \cqb)$, 
the scheme $\widetilde{D}_\sigma( \cpb; \cqb)$ is equal to the $\GL_d$-orbit of $x$, i.e. the scheme-theoretic image of the map $ \GL_d \to \Fl(d)^2$ taking $\mathbf{1}$ to $x$.  The differential of this map is the diagonal map $\Delta\colon \End H\to \End H/\Fix \pb \times \End H /\Fix \qb,$
and $T_x D_\sigma( \cpb; \cqb) = T_x \widetilde{D}_\sigma( \cpb; \cqb)= \on{im} \Delta = \ker \delta_x$ as desired.

Now the general case of (2) follows by pulling back: shrink $S$ around $x$ so that we may assume $\cH$ is trivial, and then choose a section $s\colon S\to \Fr(H)$; we get a composite map $S\to\Fr(H)\to\Fl(d)^2$ taking $x$ to $y=(\pb,\qb)$, say.  Under the differential
$ T_x S \to T_y \Fl(d)^2,$
$\ker \delta_x$ is the preimage of $\ker \delta_y = \on{im}\Delta$, by definition of $\delta_x$. (Note this does not depend on the choice of section, as verified in the definition of $\delta_x$).  And $\widetilde{D}_\sigma (\cpb;\cqb)$ is the inverse image of $\widetilde{D}_\sigma(\cvb_1;\cvb_2)\subset \Fl(d)^2$.  This verifies (2) in general.

Now (1) follows from (2) by observing that if $x \in D_\sigma(\cpb;\cqb)$, then $x \in \widetilde{D}_{\sigma'} (\cpb;\cqb)$ for some 
$\sigma' \leq \sigma$ 
in Bruhat order, so since $D_{\sigma'}(\cpb;\cqb)$ is a subscheme of $D_\sigma(\cpb;\cqb)$,
$$T_x D_\sigma(\cpb; \cqb) \supseteq T_x D_{\sigma'} (\cpb;\cqb) = \ker \delta_x(\cH; \cpb, \cqb).$$
\end{proof}

\begin{lemma} \label{lem:versalityCriterion2}
If $x \in \widetilde{D}_\sigma(\cpb;\cqb)$, where $\cpb,\cqb$ are complete flags, then the following are equivalent.
\begin{enumerate}
\item The pair $(\cP^\bu, \cQ^\bu)$ is versal at $x$.
\item The point $x$ is a smooth point of $S$ and $\delta_x$ is surjective.
\item The point $x$ is a smooth point of both $S$ and $D_\sigma(\cP^\bu;\cQ^\bu)$ and the local codimension of $D_\sigma(\cP^\bu;\cQ^\bu)$ in $S$ is equal to $\inv(\omega \sigma)$.
\end{enumerate}
\end{lemma}

\begin{proof}
We first observe that in the two-flag case, we have the following isomorphism.
$$M(H_x; P^\bu_x, Q^\bu_x) \cong \End H_x / \left( \Fix P^\bu_x + \Fix Q^\bu_x \right)$$
Together with Fact \ref{fact:invFix} and the assumption that $\sigma$ is the permutation associated to $P_x^\bu,Q^\bu_x$, this implies that $\dim M(H_x; P^\bu_x, Q^\bu_x) = \inv(\omega \sigma)$. By Lemma \ref{lem:tspaceSSigma}, we deduce that 
$$ \dim T_x D_\sigma(\cP^\bu; \cQ^\bu) = \dim T_x S - \inv(\omega \sigma) + \dim \mathrm{coker} \delta_x.$$

We now prove the Lemma. The equivalence of (1) and (2) is part of Proposition~\ref{def:versalityCriterion}, so it suffices to prove that (2) is equivalent to (3). Assume that $x$ is a smooth point of $S$ (since this is a hypothesis of both statements). Observe that
$$
\dim_x S - \inv(\omega \sigma) \leq \dim_x D_\sigma(\cP^\bu; \cQ^\bu) \leq \dim T_x D_\sigma(\cP^\bu; \cQ^\bu) = \dim_x S - \inv(\omega \sigma) + \dim \mathrm{coker} \delta_x.
$$
The first inequality follows from the local description of $D_\sigma(\cP^\bu, \cQ^\bu)$ as the inverse image of a Schubert variety (Remark \ref{rem:dsFromSchubert}). Now, $\delta_x$ is surjective if and only if $\dim \mathrm{coker} \delta_x = 0$, which holds if and only if both inequalities above hold with equality. This in turn is equivalent to $D_\sigma(\cP^\bu; \cQ^\bu)$ having local codimension $\inv(\omega \sigma)$ (first inequality) and $x$ being a smooth point of it (second inequality). So indeed (2) is equivalent to (3).
\end{proof}

\begin{example}
A simple example of all conditions in Lemma~\ref{lem:versalityCriterion2} being satisfied is two points in $\mathbb{P}^1$ moving above a smooth $1$-parameter base $S$, which come together  over a reduced point $x$ of $S$.
\end{example}

\section{A Knutson-Woo-Yong theorem for degeneracy loci of versal flags} \label{sec:KWY}

We turn our attention to intersections of degeneracy loci defined with respect to versal flags. We show that the singularities of these loci are completely controlled by the singularities of the individual degeneracy loci, and in turn by Schubert varieties. More precisely, we prove the following analog of the main theorem of \cite{knutson-woo-yong-singularities}. The results of \cite{knutson-woo-yong-singularities} concern general Schubert varieties, whereas we are concerned only with Schubert varieties of flag varieties. In the flag variety case, our result provides a generalization to $\ell$-fold intersections, as well as to the relative setting.

\begin{thm} \label{thm:kwy}
Let $P$ be an \'etale-local property of finite-type $k$-schemes that is preserved by products with affine space. Suppose that there is an integer $\ell$ and a function $f_{P,\ell}$ such that for any finite-type $k$-schemes $X_1, \cdots, X_\ell$ and point $x \in \prod X_i$,
$$P\left( x, \prod X_i \right) = f_{P,\ell}\left( P(\pi_1(x),X_1), \cdots P(\pi_\ell(x), X_\ell \right)).$$
Let $\cvb, \cpb_1, \cpb_2, \cdots, \cpb_\ell$ be a versal $(\ell+1)$-tuple of flags in a rank-$d$ vector bundle $\cH$ on a smooth variety $S$, $\sigma_1,\cdots,\sigma_\ell \in S_d$, and $x$ a point in $D_{\sigma_1,\cdots,\sigma_\ell}(\cvb; \cpb_1,\cdots,\cpb_\ell)$. Then
$$P(x,D_{\sigma_1,\cdots,\sigma_\ell}(\cvb; \cpb_1,\cdots,\cpb_\ell)) = f_{P,\ell}( P(x, D_{\sigma_1}(\cvb; \cpb_1)), \cdots, P(x, D_{\sigma_\ell}(\cvb; \cpb_\ell))).$$
\end{thm}
\noindent Throughout this section, assume that we have fixed a property $P$, denoted $P(x,X)$ for a point $x\in X$, and a function $f_{P,\ell}$ satisfying the hypotheses of Theorem \ref{thm:kwy}.

\begin{defn} \label{def:eqViaSmooth}
Let $X,Y$ be schemes with points $x \in X, y \in Y$. Let $\cH,\mathcal{J}$ be rank-$d$ vector bundles on $X,Y$ respectively, let $\cpb_1, \cdots, \cpb_\ell$ be flags in $\cH$ and let $\cqb_1,\cdots, \cqb_\ell$ be flags in $\mathcal{J}$.

We say that $(x,\cpb_1\,\cdots,\cpb_\ell)$ is \emph{equivalent via smooth morphisms} to $(y, \cqb_1,\cdots,\cqb_\ell)$ if there is a scheme $Z$ with rank-$d$ vector bundle $\mathcal{K}$, two smooth morphisms $\pi: Z \rightarrow X, \rho: Z \rightarrow Y$ and a point $z \in Z$ such that $\pi(z) = x, \rho(z) = y$, $\mathcal{K} \cong \pi^\ast \cH \cong \rho^\ast \mathcal{J}$, and such that upon identifying both pullbacks with $\mathcal{K}$, we have $\pi^\ast \cpb_i = \rho^\ast \cqb_i$ for all $i$.
\end{defn}

Equivalence via smooth morphisms is an equivalence relation. Reflexivity and symmetry are clear, and transitivity follows from standard facts about fiber products of smooth morphisms.

\begin{lemma} \label{lem:PviaSmooth}
If $(x, \cvb, \cpb_1, \cdots, \cpb_\ell)$ is equivalent via smooth morphisms to $(y, \cwb, \cqb_1, \cdots, \cqb_\ell)$, and $\sigma_1, \cdots, \sigma_\ell$ are permutations such that 
$x \in D_{\sigma_1,\cdots,\sigma_\ell}(\cvb; \cpb_1, \cdots, \cpb_\ell)$,
then 
$$y \in D_{\sigma_1,\cdots,\sigma_\ell}(\cwb; \cqb_1, \cdots, \cqb_\ell)$$
and
$$P(x, D_{\sigma_1,\cdots,\sigma_\ell}(\cvb; \cpb_1, \cdots, \cpb_\ell)) = P(y, D_{\sigma_1,\cdots,\sigma_\ell}(\cwb; \cqb_1, \cdots, \cqb_\ell)).$$
Furthermore, the codimension at $y$ of $D_{\sigma_1,\cdots,\sigma_\ell}(\cwb; \cqb_1, \cdots, \cqb_\ell)$ in $Y$ is equal to the codimension at $x$ of $D_{\sigma_1,\cdots,\sigma_\ell}(\cvb; \cpb_1, \cdots, \cpb_\ell)$ in $X$.
\end{lemma}

\begin{proof}
Let $\pi: Z \rightarrow X$, $\rho\colon Z\to Y$, and $z \in Z$ be as in Definition \ref{def:eqViaSmooth}. For simplicity, abbreviate $D_{\sigma_1,\cdots,\sigma_\ell}(\cvb; \cpb_1,\cdots,\cpb_\ell)$ by $D_X$, $D_{\sigma_1,\cdots,\sigma_\ell}(\cwb; \cqb_1,\cdots,\cqb_\ell)$ by $D_Y$, and $D_{\sigma_1,\cdots,\sigma_\ell}(\pi^\ast \cvb; \pi^\ast \cpb_1,\cdots,\pi^\ast \cpb_\ell)$ by $D_Z$. Observe that $D_Z = \pi^{-1} (D_X) = \rho^{-1}(D_Y)$. Since $x \in D_X$, it follows that $z \in D_Z$ and $y \in D_Y$.

The restriction $D_Z \rightarrow D_X$ of $\pi$ is smooth. Let $n$ be the relative dimension of this morphism at $z$. Then $n$ is also the relative dimension of $\pi$ at $z$; it follows that the codimension of $D_Z$ in $Z$ is equal to the codimension of $D_X$ at $x$. 

There exist affine neighborhoods $z \in U \subseteq D_Z$ and $x \in V \subseteq D_X$ and an \'etale morphism $e: U \rightarrow \mathbb{A}^n_V$ such that $U \rightarrow V$ factors through the projection $\mathbb{A}^n_V \rightarrow V$ \cite[{Tag 039P}]{stacks-project}. Since the property $P$ is \'etale-local and unaffected by products with affine space, it follows that $P(z,D_Z) = P(e(z), \mathbb{A}^n_V) = P(x, D_X)$. 

Applying the same logic to $\rho$, it follows that the codimension of $D_Y$ in $Y$ is equal to the codimension of $D_Z$ in $Z$, and $P(z,D_Z) = P(y,D_Y)$. The result follows.
\end{proof}

\begin{proof}[Proof of Theorem \ref{thm:kwy}]
The frame bundle $\Fr(\cH) \rightarrow S$ is surjective, so $x \in S$ lifts to a point $x' \in \Fr(\cH)$. The versal tuple $(\cvb, \cpb_1, \cdots, \cpb_\ell)$ determines a smooth morphism $p: \Fr(\cH) \rightarrow \Fl(d)^{\ell+1}$. Next, apply the same construction, with $S$ replaced by $\Fl(d)^\ell$, with $T = k^d \times \Fl(d)^\ell$ being the trivial rank $d$ vector bundle, and 
with flags $(F^\bu, \cvb_1,\cdots, \cvb_\ell)$, where $\{\cvb_i\}$ are the tautological flags and $F^\bu$ is an arbitrary fixed flag.  
We obtain a morphism $v: \Fr(T) \rightarrow \Fl(d)^{\ell+1}$. The tuple $(F^\bu, \cvb_1,\cdots, \cvb_\ell)$ is versal, by $\ell$ applications of Lemma \ref{lem:versalityCriterion3} (and the straightforward observation that a single flag is versal), so $v$ is smooth. It is straightforward to check that $v$ is also surjective, so $p(x') \in \Fl(d)^{\ell+1}$ lifts to $x'' \in \Fr(T)= \GL_d \times \Fl(d)^\ell$. Denote by $y$ the image of $x''$ in $\Fl(d)^\ell$. Denote by $\pi_1,\cdots,\pi_\ell$ the projection maps from $\Fl(d)^\ell$ to $\Fl(d)$. The maps and points constructed are summarized in Figure \ref{fig:staircase}. In this diagram, all morphisms are smooth.

\begin{figure}[h!] 
\begin{tikzcd}[column sep = tiny]
& x' \in \Fr(\cH) \arrow[dl, two heads] \arrow[dr,"p"{above}] 
& & 
x'' \in \Fr(T) \arrow[dl, two heads, "v"{above}] \arrow[dr]& & \\
x \in S & & \Fl(d)^{\ell+1} & & y \in \Fl(d)^\ell \arrow[dr, "\pi_i"{above}] & \\
& & & & & \Fl(d)\\
\end{tikzcd}
\caption{  The morphisms and chosen points in the proof of Theorem \ref{thm:kwy}.}
\label{fig:staircase}
\end{figure}
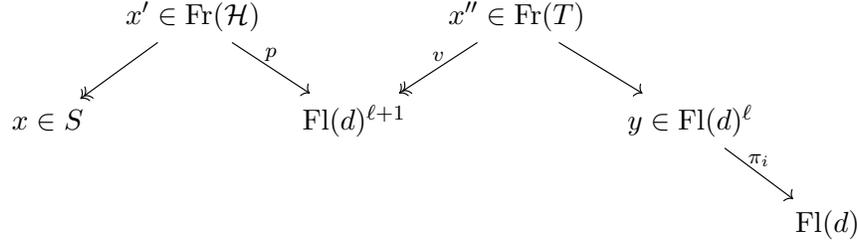
The intuition behind the first four arrows of the diagram is that a family of $\ell+1$ versal flags can, after coordinate change, be regarded as a family of $\ell$ versal flags together with one fixed flag.

Denote the tautological flags in the trivial bundle on $\Fl(d)^{\ell+1}$ by $\mathcal{U}^\bu_1,\cdots, \mathcal{U}^\bu_{\ell+1}$. By construction, $p^\ast \mathcal{U}^\bu_1, \cdots, p^\ast \mathcal{U}^\bu_{\ell+1}$ are equal to the pullbacks from $S$ of $\cvb, \cpb_1, \cdots, \cpb_\ell$, respectively. Similarly, $v^\ast \mathcal{U}^\bu_1, \cdots, v^\ast \mathcal{U}^\bu_{\ell+1}$ are equal to the pullbacks from $\Fl(d)^{\ell+1}$ of $F^\bu, \cvb_1, \cdots, \cvb_\ell$, respectively. Finally, denoting the tautological and trivial flags over $\Fl(d)$ by $\cub, E^\bu$ respectively, we have $\cvb_i = \pi_i^\ast \cub$ and $F^\bu = \pi_i^\ast E^\bu$. From this, we deduce the following equivalences via smooth morphisms.

\begin{eqnarray*}
(x, \cvb, \cpb_1, \cdots, \cpb_\ell) &\sim& (y, F^\bu, \cvb_1, \cdots, \cvb_\ell)\\
(x, \cvb, \cpb_i) &\sim& (\pi_i(y), E^\bu, \cub)
\end{eqnarray*}

Observe that for each $i$, $D_{\sigma_i}(F^\bu; \cvb_i) = \pi^{-1}_i ( D_{\sigma_i}(E^\bu; \cub))$ inside $\Fl(d)^\ell$. Therefore
\begin{eqnarray*}
D_{\sigma_1, \cdots, \sigma_\ell}(F^\bu; \cvb_1, \cdots, \cvb_\ell) &=& \bigcap_i \pi_i^{-1} (D_{\sigma_i}(E^\bu; \cub))\\
&\cong& \prod_i D_{\sigma_i}(E^\bu; \cub).
\end{eqnarray*}
Putting this together and applying Lemma \ref{lem:PviaSmooth}:
\begin{eqnarray*}
P(x, D_{\sigma_1, \cdots, \sigma_{\ell}}(\cvb; \cpb_1, \cdots, \cpb_\ell)) &=& P(y, D_{\sigma_1,\cdots,\sigma_\ell}(F^\bu; \cvb_1,\cdots, \cvb_\ell))\\
&=& f_{P,\ell}\Big(
P(\pi_1(y), D_{\sigma_1}(E^\bu; \cub)), \cdots, P(\pi_\ell(y), D_{\sigma_\ell}(E^\bu; \cub)) 
\Big)\\
&=& f_{P,\ell} \Big(
P(x, D_{\sigma_1}(\cvb; \cpb_1)), \cdots, P(x, D_{\sigma_\ell}(\cvb; \cpb_\ell))
\Big).
\end{eqnarray*}
\end{proof}

In fact, the proof of Theorem \ref{thm:kwy} shows that we can say slightly more: 
we can reduce completely to Schubert varieties in flag varieties.

\begin{thm} \label{thm:kwy2}
With the same hypotheses as Theorem \ref{thm:kwy}, let 
$\sigma'_1 \leq \sigma_1, \cdots, \sigma'_\ell \leq \sigma_\ell$ 
be permutations such that $x \in \widetilde{D}_{\sigma'_i}(\cvb; \cpb_i)$ for all $i$. Let $z_i$ be any point in the open Schubert cell $\widetilde{X}_{\sigma'_i} \subseteq \Fl(d)$. Then
$$
P(x,D_{\sigma_1, \cdots, \sigma_\ell}(\cvb; \cpb_1, \cdots, \cpb_\ell)) = f_{P,\ell}( P(z_1, X_{\sigma_1}), \cdots, P(z_\ell, X_{\sigma_\ell})).
$$
\end{thm}

\begin{proof}
In the notation of the proof of Theorem \ref{thm:kwy}, it suffices to prove that $$P(\pi_i(y), D_{\sigma_i}(E^\bu; \cvb)) = P(z_i, X_{\sigma_i}).$$ 
Observe that $D_{\sigma_i}(E^\bu; \cub) = D_{\sigma_i^{-1}}(\cub; E^\bu) = X_{\sigma_i^{-1}}(E^\bu)$. 
Similarly, $D_{\sigma_i'}(E^\bu,\cub) = X_{(\sigma_i')^{-1}}(E^\bu)$. Since 
Schubert cells are Borel orbits,
it follows that for any $w \in \widetilde{X}_{(\sigma_i')^{-1}}(E^\bu)$, 
$$P(\pi_i(y), D_{\sigma_i}(E^\bu; \cvb)) = P(w, X_{\sigma_i^{-1}}).$$ 
To see that the same result holds when $\sigma_i^{-1}, (\sigma_i')^{-1}$ are replaced with $\sigma_i, \sigma'_i$, observe that we may apply the entire argument to the case $S = \Fl(d), x = w, \ell = 1, \cvb =  \cub, \cpb_1 = E^\bu$, from which it follows that for any $z_i \in \widetilde{X}_{\sigma'_i}$, $P(w,X_{\sigma_i^{-1}}(E^\bu)) = P(z_i, X_{\sigma_i}(E^\bu))$, and the result follows.
\end{proof}

We also point out another useful consequence of the proof of Theorem \ref{thm:kwy}.

\begin{prop} \label{prop:codimViaKWY}
If $\cvb, \cpb_1, \cdots, \cpb_\ell$ is a versal $(\ell+1)$-tuple of flags in a vector bundle $\cH$ over $S$, then
$D_{\sigma_1,\cdots,\sigma_\ell}(\cvb; \cpb_1, \cdots, \cpb_\ell)$ has pure codimension $\inv(\omega \sigma_1) + \cdots + \inv(\omega \sigma_\ell)$ in $S$.
\end{prop}

\begin{proof}
By the codimension statement in Lemma \ref{lem:PviaSmooth} and the equivalence via smooth morphisms in the proof of Theorem \ref{thm:kwy}, we may assume that $S = \Fl(d)^\ell$, $\cpb_i = \mathcal{U}^\bu_i$, and $\cvb = F^\bu$. The result now follows from the fact that, in $\Fl(d)$, the codimension of $D_{\sigma_i}(\cvb; F^\bu)$ is equal to the codimension of $X_{\sigma_i}$, which is $\inv(\omega \sigma_i)$ (see Section \ref{ssec:schubert}).
\end{proof}

Compare Corollary~\ref{cor:kwy-gen} below with \cite[Corollaries 1.2, 1.3, 1.4, 3.1]{knutson-woo-yong-singularities}; note also the generalization from $\ell=2$ to any $\ell$.
\begin{cor}\label{cor:kwy-gen}
By making various choices of $P$ and  $f_{P,\ell}$, we deduce the following about any degeneracy locus $D = D_{\sigma_1, \cdots, \sigma_\ell}$ for versal flags $\cvb,\cpb_1,\ldots,\cpb_\ell$.
\begin{enumerate}
    \item 
    The smooth locus of $D$ is the intersection of the smooth loci of each $D_{\sigma_i}$. 
    \item $D$ is normal and Cohen-Macaulay.
    \item $D$ is Gorenstein at $x\in D$ if and only if each $X_{\sigma_i}$ is Gorenstein along $\widetilde{X}_{\sigma_i'}$, where $\sigma_i'$ is as in the hypotheses of Theorem~\ref{thm:kwy2}. 
    \item Recall that for $x\in X$, the $H$-polynomial $H_{x,X}(q)$ is defined by the equation
    $$ \mathrm{Hilb}(G_{\mathfrak{m}_x}(\cO_{X,x});q) = \frac{H_{x,X}(q)}{(1-q)^{\dim_x X}},$$
    where $G_{\mathfrak{m}_x}(\cO_{X,x})$ denotes the associated graded ring of $\cO_{X,x}$.  Recall the Hilbert-Samuel multiplicity is ${\rm mult}_{x,X} = H_{x,X}(1)$.  Then
    $$ H_{x,D}(q) = \prod_i H_{x,D_{\sigma_i}}(q),$$
    and hence
    $ {\rm mult}_{x,D} = \prod_i {\rm mult}_{x,D_{\sigma_i}}.$
\end{enumerate}
\end{cor}

\begin{proof}
For (1), (2) and (3), recall that Cohen-Macaulayness, reducedness, 
normalness, regularity, and being Gorenstein are all \'etale-local \cite[{Tag 025L, Tag 0E12}]{stacks-project}; moreover all Schubert varieties are normal and Cohen-Macaulay. In all three cases, we take the function $f_{P,\ell}$ to be ``logical and.''  For (4), the associated graded of $\cO_{D,x}$ can be computed from its completion, and \'etale morphisms induce isomorphisms on completed local rings (recall we always work over an algebraically closed field).  Hence the $H$-polynomial is also an \'etale-local invariant. The result follows by taking $f_{P,\ell}$ to be the usual multiplication of polynomials. 
\end{proof}

See \cite{woo-yong-when} for a characterization of which Schubert varieties, in type A, are Gorenstein in terms of permutation pattern-avoidance, and a conjectured general characterization of the non-Gorenstein locus. The conjecture is proven by Perrin for minuscule Schubert varieties \cite{perrin-gorenstein}; it is open in general as far as we know.

We can now prove Theorem \ref{thm:kwyIntro} from the introduction. This requires a short argument applying Theorem \ref{thm:kwy} to cases where one defining flag may not be a complete flag.

\begin{proof}[Proof of Theorem \ref{thm:kwyIntro}]
Let $i_0, \cdots, i_s$ be the coranks of the nests of sets. Let $F = \Fl(i_0, \cdots, i_s; \cH)$ and $F' = \Fl(H)$ (the complete flag variety); denote the projections to $S$ by $\pi, \pi'$ respectively. Let $\cvb$ be the tautological complete flag in $\pi'^\ast \cH$. Since each essential set $\Ess(\sigma(\ab_i))$ is contained in $\{i_0, \cdots, i_s\} \times [d]$, we see by comparing defining rank conditions that the inverse image of $X_{\ab_1}(\cpb_1) \cap \cdots X_{\ab_\ell}(\cpb_\ell)$ in $F'$ is equal to $D_{\sigma(\ab_1), \cdots, \sigma(\ab_\ell)} ( \cvb; \pi'^\ast \cpb_1, \cdots, \pi'^\ast \cpb_\ell)$, to which Theorem \ref{thm:kwy} applies. The Theorem now follows from the fact that the forgetful map $F' \rightarrow F$ is a fiber bundle, with fibers \'etale-locally isomorphic to affine space.
\end{proof}

\section{Relative Richardson varieties}
\label{sec:rr}

\subsection{Definitions}\label{ssec:relrich}

We now define relative Richardson varieties and deduce their basic properties from the results of the previous section.  In particular we take $\ell=2$ in this section, as we are not aware of a generalization of the cohomological arguments below to higher  $\ell$. 

\begin{defn} \label{loc:rab}
Let $S$ be an irreducible smooth variety, with a rank $d$ vector bundle $\cH$ and two complete versal flags $\cpb,\cqb$. Let $0=i_0 < i_1 < \cdots < i_s = d$ be integers. For any two nests of sets $\ab,\bb$ with coranks $i_0, \cdots, i_s$, define the subvariety of $\Fl(i_0,\ldots,i_s;\cH)$
$$\Rab(\cpb, \cqb) = X_{\ab}(\cpb) \cap X_{\bb}(\cqb),$$
where $X_{\ab}(\cpb)$ and $X_{\bb}(\cqb)$ are defined in~\S\ref{ssec:degen}. Such a variety $\Rab$ is called a \emph{relative Richardson variety over $S$.}

For any two permutations $\sigma,\tau$, we also write $\Rst(\cpb,\cqb)$ \label{loc:rst} as alternate notation for $\Rab(\cpb,\cqb)$, where $\ab,\bb$ are complete nests of sets with decreasing completions $\sigma,\tau$, respectively. We write $\widetilde{R}_{\ab,\bb}$ for the open subscheme where all the defining rank conditions hold with equality, and use the notation $\widetilde{R}_{\sigma,\tau}$ similarly.
\end{defn}

We emphasize that we reserve the phrase ``relative Richardson variety'' for situations where $S$ is smooth and irreducible and the flags $\cpb,\cqb$ are versal, so that relative Richardson varieties share the geometric properties enjoyed by Richardson varieties, as summarized in Theorem \ref{thm:relativeRichardson}.

\exnow{
Let $S=\Spec k$, let $d=5$, and 
$$A^\bu  =  \{0,1,2,3,4\} \supset \{0,2,4\} \supset \emptyset,\quad
B^\bu = \{0,1,2,3,4\} \supset \{0,1,2\} \supset  \emptyset .$$
Then $\Rab$ is isomorphic to a Schubert variety with respect to the flag $P^\bu$, parametrizing 2-dimensional subspaces $V^3$ with 
$$
\dim V^3\cap P^2 \ge 2 \quad\text{and}\quad
\dim V^3\cap P^4 \ge 1.
$$
}

The cohomological statements in Theorem \ref{thm:relativeRichardson} will be proved in Section \ref{ssec:coh}. The rest of Theorem \ref{thm:relativeRichardson} follows readily from the results of the previous section, as summarized below.

\begin{thm} \label{thm:RabLocal}
A relative Richardson variety $\Rab$ is normal and Cohen-Macaulay, of pure codimension $\inv(\omega \sigma(\ab)) + \inv(\omega \sigma(\bb))$ in $\Fl(\cH; i_0, \cdots, i_s)$, and the smooth locus of $\Rab$ is equal to the intersection of the smooth loci of the relative Schubert varieties $X_{A^\bu}$ and $X_{B^\bu}$.  The open subscheme $\widetilde{R}_{\ab,\bb}$ is dense in the smooth locus of $\Rab$.
\end{thm}

\begin{proof}
Let $\sigma = \sigma(\ab)$ and $\tau = \sigma(\bb)$. Observe that $\Rst = f^{-1} ( \Rab)$, where $f: \Fl(\cH) \rightarrow \Fl(i_0, \cdots, i_s; \cH)$ is the forgetful morphism from the complete flag variety. Since $f$ is a fiber bundle with smooth irreducible fibers, we see that it suffices to prove the theorem for $\Rst$, i.e., for the case of complete flags. Note that
$$\Rst = \Dst( \cvb; \pi^\ast \cpb, \pi^\ast \cqb),$$
where $\pi: \Fl(\cH) \rightarrow S$ is the structure map and $\cvb$ is the tautological bundle of $\Fl(\cH)$.  The flags $\cvb,\pi^\ast \cpb,$ and  $\pi^\ast \cqb$ are versal by Lemma \ref{lem:versalityCriterion3}, so Proposition \ref{prop:codimViaKWY} and Corollary \ref{cor:kwy-gen} imply that $\Rst$ is normal and Cohen-Macaulay of pure codimension $\inv(\omega \sigma)+\inv(\omega \tau)$ in $\Fl(\cH)$, with smooth locus equal to the intersection of the smooth loci of $\Ds(\cvb; \pi^\ast \cpb) = X_\sigma(\cpb)$ and $D_\tau(\cvb; \pi^\ast \cqb) = X_\tau(\cqb)$. 
Theorem \ref{thm:kwy2} and the description in Section \ref{ssec:schubert} of the singular locus of Schubert varieties show that the singular locus of $\Rst$ is a union of certain loci $R_{\sigma',\tau'}$ where $\sigma' < \sigma$ or $\tau' < \tau$, and Proposition \ref{prop:codimViaKWY} shows that each such locus has positive codimension. In particular, $\widetilde{R}_{\sigma,\tau}$ is dense and contained in the smooth locus of $\Rst$.
\end{proof}

\subsection{Cohomology of relative Richardson varieties} \label{ssec:coh}

In order to understand the cohomology of the schemes $\Rab$, we will relate them to each other, and to the base $S$, via morphisms for which the total pushforward of the structure sheaf is trivial. Call a proper morphism of $k$-schemes $\pi\colon X\to Y$ {\em $\cO$-connected} if $\cO_Y \to \pi_* \cO_X$ is an isomorphism, and call $\pi$ a {\em cohomological equivalence} if it is $\cO$-connected and  
$$R^i \pi_* \cO_X = 0 \quad\text{for all } i>0.$$  
We note that the term ``cohomological equivalence'' is sometimes used in the literature specifically for birational morphisms (e.g. in \cite{kovacs-rational}), but we use it in a more general way.
Note (see e.g.~\cite[Exercise 8.1]{hartshorne-algebraic}) that if $\pi$ is a cohomological equivalence then it induces canonical isomorphisms $H^i(X,\cO_X)\cong H^i (Y,\cO_Y)$
for all $i\ge 0$.  In particular, we have $\chi(X,\cO_X) = \chi(Y,\cO_Y)$
in this situation.  In what follows, we will make use of the following fact, which can be deduced from the Grothendieck spectral sequence: if $f\colon X' \to X$ is a cohomological equivalence and $\pi\colon X \to Y$ is any proper morphism, then $\pi$ is a cohomological equivalence if and only if $\pi\circ f$ is a cohomological equivalence. The purpose of this subsection is to prove the following theorem.

\begin{thm} \label{thm:rabgood}
Let $\cP^\bu,\cQ^\bu$ be versal flags in a rank-$d$ vector bundle $\cH$ on a scheme $S$, and let $\ab, \bb$ be any two nests of sets as defined in \S\ref{ssec:relrich}. Denote by $R_{\ab, \bb}$ the resulting relative Richardson variety.
Let $\sigma = \sigma(\ab)$ and $\tau = \sigma(\bb).$
Then the image of the morphism $\Rab \rightarrow S$ is $D_{\tau \star \sigma^{-1}}(\cP^\bu; \cQ^\bu)$, and the morphism $\Rab \rightarrow D_{\tau \star \sigma^{-1}}(\cP^\bu; \cQ^\bu)$ is a cohomological equivalence. 
\end{thm}

Our strategy for proving Theorem \ref{thm:rabgood} is as follows. We first reduce to the case of complete flag varieties, i.e., varieties $\Rst$ (Corollary \ref{cor:sigmaGood}). We then show that the desired statement about $\Rst$ can be deduced from the statement for $R_{\sigma', \tau'}$, where $\sigma',\tau'$ is another pair of permutations with $\inv(\sigma') < \inv(\sigma)$ and $\tau' \star \sigma'^{-1} = \tau \star \sigma^{-1}$ (Lemma \ref{lem:goodMoves}). This reduces by induction to the case of the morphism $R_{\mathrm{id}, \tau \star \sigma^{-1}} \to S$, which we show (Lemma \ref{lem:RtoS}).

Throughout the section, we fix the choice of $S, \cH, \cP^\bu, \cQ^\bu$. We do not assume in general that $\cP^\bu, \cQ^\bu$ are versal, as several auxiliary results do not require this hypothesis; we will state it specifically when it is needed.
We begin with three useful criteria for cohomological equivalences.

\begin{fact} \label{fact:grassBundle}
If $f: X \rightarrow Y$ is a Grassmannian bundle, then $f$ is a cohomological equivalence.
\end{fact}

\begin{proof}
This follows from the fact that the structure sheaf of the Grassmannian variety $\Gr(t,n)$ has no higher cohomology. This in turn follows from the Borel-Weil-Bott theorem, or can be seen more directly by induction on $t$, by observing that both the forgetful morphisms $\Fl(t-1,t;n) \to \Gr(t-1,n)$ and $\Fl(t-1,t;n) \to \Gr(t,n)$ are cohomological equivalences.
\end{proof}

\begin{fact} \label{fact:oconn}
If $f: X \rightarrow Y$ is a birational morphism of normal, irreducible, projective varieties, then $f$ is $\cO$-connected.
\end{fact}

\begin{proof}
See the proof of \cite[Corollary 11.4]{hartshorne-algebraic}.
\end{proof}

\begin{fact}\label{brion}
Suppose $X\xrightarrow{\pi} Y$ factors as $X\xrightarrow{j} Z\xrightarrow{\pi'} Y$, where $X\to Z$ is a closed immersion with ideal sheaf $\mathcal{I}$ and $Z\to Y$ is a $\mathbb{P}^1$-bundle. Then $R^i \pi_* \cO_X = 0$ for all $i>0.$
\end{fact}

\begin{proof}  This argument may be found in \cite[\S2.1]{brion-lectures}; we summarize it here for convenience. The statement holds for $i>1$ since all fibers of $\pi$ have dimension at most 1.  For $i=1$, the exact sequence $0\to \mathcal{I} \to \cO_Z \to j_* \cO_X \to 0$ yields the following portion of a long exact sequence.
$$\cdots \to R^1\pi'_* \cO_Z \to R^1 \pi_*\cO_X \to R^2 \pi'_*\mathcal{I} \to \cdots$$
The first term is $0$ since $\pi'$ is a $\mathbb{P}^1$-bundle, and the third term is $0$ since all fibers of $\pi'$ have dimension 1. So the middle term is $0$.
\end{proof}

In what follows, we fix the following notation. If $A,B \subseteq \ZZ$ are finite sets, we write $A < B$ to mean $\max A < \min B$. If $\ab, \bb$ are two nests of sets, and $j$ is an index with $0 < j < s$, we denote the nested sets obtained by removing the $j$th set as follows.

$$A_{j}^\bu = (A^{i_0} \supset \cdots \supset \widehat{A^{i_j}} \supset \cdots \supset A^{i_s}), \qquad
B_{j}^\bu = (B^{i_0} \supset \cdots \supset \widehat{B^{i_j}} \supset \cdots \supset B^{i_s}).$$

\begin{lemma} \label{lem:GrassBundle2}
Let $S$ be a smooth variety, with a rank $d$ vector bundle $\cH$ and two complete flags $\cpb,\cqb$.
Suppose that $\ab, \bb, j$ are as above, and that we have both $$(A^{i_{j-1}}\setminus A^{i_j}) >(A^{i_j}\setminus A^{i_{j+1}}) \text{ and }
(B^{i_{j-1}}\setminus B^{i_j}) > (B^{i_j}\setminus B^{i_{j+1}}).$$
Then $R_{\ab,\bb} \rightarrow R_{\ab_j, \bb_j}$ is a cohomological equivalence, and it is a fiber bundle with smooth irreducible fibers.
\end{lemma}

\noindent For example, let $\ab = \bb = (\{0,1,2,3\}  \supset \{0,1\} \supset \emptyset)$, and $j=1$.  Then $R_{\ab,\bb} \rightarrow R_{A_j^\bu, B_j^\bu}$ is a Grassmannian $\mathbb{G}(1,\PP^3)$ over a point. 

\begin{proof}[Proof of Lemma~\ref{lem:GrassBundle2}]
Observe that $\Ess(\sigma(\ab),\sigma(\bb))$ does not contain $(i_j, b)$ for any value of $b$ (see Remark \ref{rem:essentialDescents}). So $\Rab$ and $R_{A_j^\bu, B_j^\bu}$ are defined by the same rank conditions, and thus $\Rab$ is the inverse image of $R_{A_j^\bu, B_j^\bu}$ under a forgetful morphism of partial flag varieties. The Lemma now follows from Fact~\ref{fact:grassBundle}.
\end{proof}

Applying Lemma \ref{lem:GrassBundle2} repeatedly gives the following.

\begin{cor} \label{cor:sigmaGood}
Let $\cP^\bu, \cQ^\bu$ be complete flags in a vector bundle $\cH$ on a smooth scheme $S$. For any choice of $\ab, \bb$,
the forgetful morphism
$$R_{\sigma(\ab), \sigma(\bb)} \rightarrow R_{\ab, \bb}$$
is a cohomological equivalence with smooth irreducible fibers.
\end{cor}

\begin{lemma}\label{lem:board-l4}
Let $\cP^\bu, \cQ^\bu$ be versal flags in a vector bundle $\cH$ on a smooth scheme $S$. 
Given $\ab, \bb$, define $\ab_j$ and $\bb_j$ as in Lemma \ref{lem:GrassBundle2}. Suppose further that $i_{j+1} = i_{j-1}+2$, and either
 $$(A^{i_{j-1}}\setminus A^{i_j}) >(A^{i_j}\setminus A^{i_{j+1}}) \text{ or }
(B^{i_{j-1}}\setminus B^{i_j}) > (B^{i_j}\setminus B^{i_{j+1}}).$$
Then the forgetful morphism $\pi\colon \Rab \to R_{\ab_j, \bb_j}$ is a cohomological equivalence.
\end{lemma}

\begin{proof}
If both of these conditions hold, then the result follows from Lemma \ref{lem:GrassBundle2}. Therefore we may assume without loss of generality that $i_{j+1} = i_{j-1}+2$ and that 
 $$(A^{i_{j-1}}\setminus A^{i_j}) >(A^{i_j}\setminus A^{i_{j+1}}) \text{ and }
(B^{i_{j-1}}\setminus B^{i_j}) < (B^{i_j}\setminus B^{i_{j+1}}).$$
Writing $\{b\} = B^{i_{j-1}}\setminus B^{i_{j}}$ and $\{c\} = B^{i_{j}}\setminus B^{i_{j+1}}$, we have $b<c$. Define $C^\bu$ by ``exchanging $b$ and $c$,''  in other words $C^{i_j} = C^{i_{j+1}} \cup \{b\}$ and $C^\bu$ is otherwise the same as $\bb$. Notice that $\pi$ factors as $\Rab \to R_{\ab,C^\bu} \to \Rabj,$ where the first morphism is a closed immersion, and the second is the structure map for a $\PP^1$-bundle,
by the argument in Lemma \ref{lem:GrassBundle2}.  This implies that $R^i \pi_* \cO_{\Rab} = 0$ for $i>0$ (Fact~\ref{brion}). It remains to verify that $\pi$ is $\cO$-connected. The statement is local in $R_{\ab_j, \bb_j}$, so we assume that $R_{\ab_j,\bb_j}$ is connected. Since it is also normal (Theorem \ref{thm:RabLocal}), we may assume that $R_{\ab_j,\bb_j}$ is irreducible. By Fact \ref{fact:oconn}, it suffices to check that $\pi$ is birational. The open subscheme $\widetilde{R}_{\ab, \bb}$ is dense in $\Rab$ by Theorem~\ref{thm:RabLocal}, and likewise $\widetilde{R}_{\ab_j,\bb_j}$ is dense in $R_{\ab_j,\bb_j}$. At any point of $\widetilde{R}_{\ab_j,\bb_j}$, corresponding to flags $V^\bu, P^\bu, Q^\bu$ in $H$, there is a unique choice of subspace $V^{i_j}$ strictly between $V^{i_{j-1}}$ and $V^{i_{j+1}}$ which can be added to the flag $V^\bu$ to produce a point of $\Rab$.  More precisely, over $\widetilde{R}_{A^\bu_j, B^\bu_j}$, the morphism $\pi$ is a structure morphism for a relative Schubert subvariety of a $\PP^1$-bundle consisting of all relative flags in this $\PP^1$-bundle  that {\em coincide with} a fixed flag.
This shows that in fact $\pi$ restricts to an isomorphism $\widetilde{R}_{\ab, \bb} \rightarrow \widetilde{R}_{A^\bu_j, B^\bu_j}$, which shows that $\pi$ is indeed birational, which completes the proof.
\end{proof}

\begin{lemma}\label{lem:goodMoves}
Suppose $\sigma,\tau$ is a pair of permutations, and $s$ is a simple transposition. Assume that $\cP^\bu, \cQ^\bu$ are versal. Then $R_{\sigma \star s, \tau} \to S$ and $R_{\sigma, \tau \star s} \to S$ have the same image $S'$, and if one is a cohomological equivalence then so is the other.
\end{lemma}

\begin{proof}
First, consider the case where $\sigma \star s = \sigma$. We may assume that $\tau \star s \neq \tau$. Let $j$ be the index such that $s$ transposes $j$ and $j-1$. By Fact \ref{fact:demazure}, $\sigma_{j-1} > \sigma_j$, $\tau_{j-1} < \tau_j$, and $\tau \star s = \tau s$. Let $\ab,\bb,C^\bu$ be the complete nests of sets associated to $\sigma,\tau$, and $\tau s$, respectively. 

Both $R_{\sigma,\tau} = \Rab$ and $R_{\sigma,\tau s} = R_{\ab,C^\bu}$ have forgetful morphisms to $R_{\ab_j,\bb_j}$; we record the relevant morphisms in the commuting diagram below.

\begin{figure}[h!]
\begin{tikzcd}[column sep = tiny]
\Rab \arrow[rr,"i"] \arrow[dr,"f"{above}] \arrow[dddr, "\pi"{below}] & & R_{\ab,C^\bu} \arrow[dl,"f'"{above}] \arrow[dddl, "\pi'"]\\
 & R_{\ab_j,\bb_j} \arrow[dd,"\pi^j"] &\\
 & & \\
  & S&
\end{tikzcd}
\end{figure}
\noindent Here $i$ is a closed immersion, and $f'$ is a $\PP^1$-bundle.  The maps $f$ and $f'$ are known to be cohomological equivalences, by Lemmas~\ref{lem:GrassBundle2} and~\ref{lem:board-l4}. In particular, all three push-forwards $\pi_\ast \cO_{\Rab}$, $\pi^j_\ast \cO_{R_{\ab_j, \bb_j}}, \pi'_\ast \cO_{R_{\ab, C^\bu}}$ agree. It follows that the scheme-theoretic image of all three has the same ideal sheaf on $S$ (namely, the kernel of $\cO_S \rightarrow \pi_\ast \cO_{\Rab}$). Hence the image of all three morphisms $\pi, \pi^j, \pi'$ is the same; call this image $S'$.

All three morphisms $\pi, \pi^j, \pi'$ factor through $S'$. It follows from the Grothendieck spectral sequence that the morphism $\Rab\rightarrow S'$ is a cohomological equivalence if and only the morphism $R_{\ab_j,\bb_j} \rightarrow S'$ is a cohomological equivalence, if and only if the morphism $R_{\ab, C^\bu} \rightarrow S'$ is a cohomological equivalence. This establishes the result in the case where $\sigma \star s = \sigma$.

Next, the case where $\tau \star s = \tau$ follows from the first case by exchanging the flags. It remains to consider the case where $\tau \star s \neq \tau$ and $\sigma \star s \neq \sigma$. In fact, this case follows from the first two: by Fact \ref{fact:demazure} we have $\sigma \star s = \sigma \star s \star s = \sigma s$ and $\tau \star s = \tau \star s \star s = \tau s$, and the result follows by first relating $R_{\sigma \star s \star s, \tau}$ to $R_{\sigma \star s, \tau \star s}$ and then relating $R_{\sigma \star s, \tau \star s}$ to $R_{\sigma, \tau \star s \star s}$.
\end{proof}

\begin{lemma} \label{lem:RtoS}
Suppose that $\sigma \in S_d$. For any scheme $S$ with vector bundle $\cH$ and complete flags $\cP^\bu, \cQ^\bu$, the morphism $R_{\mathrm{id}, \sigma}(\cP^\bu, \cQ^\bu) \rightarrow S$ is a closed immersion with image $D_{\sigma}(\cP^\bu; \cQ^\bu)$.
\end{lemma}

\begin{proof}
Consider the functor of points of $R_{\mathrm{id},\sigma}$. A morphism $T \rightarrow R_{\mathrm{id},\sigma}$ corresponds to a morphism $t: T \rightarrow S$ and a complete flag $\cW^\bu$ of $t^\ast \cH$ such that the permutation associated to $(\cwb, t^\ast \cP^\bu)$ is at most $\mathrm{id}$ and the permutation associated to $(\cwb, t^\ast \cQ^\bu)$ is at most $\sigma$ (where these statements are meant scheme-theoretically, i.e., as determinantal loci). The former condition is equivalent to saying that $\cW^\bu$ is identical to $t^\ast \cP^\bu$. So in fact morphisms $T \rightarrow R_{\mathrm{id},\sigma}$ correspond to morphisms $t: T \rightarrow S$ such that $t^\ast \cP^\bu, t^\ast \cQ^\bu$ have associated permutation at least $\sigma$. But this is a description of the functor of $D_\sigma(\cP^\bu; \cQ^\bu)$. We deduce that  the morphism $R_{\mathrm{id},\sigma} \rightarrow S$ induces a bijection between morphisms $T \rightarrow R_{\sigma, \tau}$ and morphisms $T \rightarrow D_\sigma(\cP^\bu; \cQ^\bu)$; the result follows.
\end{proof}

We can now prove Theorems~\ref{thm:rabgood} and \ref{thm:relativeRichardson}.

\begin{proof}[Proof of Theorem~\ref{thm:rabgood}]

Let $\sigma = \sigma(\ab)$ and $\tau = \sigma(\bb)$. By Corollary \ref{cor:sigmaGood} and the Grothendieck spectral sequence, $\Rab \rightarrow S$ and $\Rst \rightarrow S$ have the same image $S'$, and one is a cohomological equivalence if and only if the other is. So it suffices to consider complete $\ab, \bb$, i.e., the schemes $\Rst$. 

We prove the result by induction on $\inv(\tau)$. The base case, $\tau = \mathrm{id}$, follows from Lemma \ref{lem:RtoS}. For the induction step, let $s$ be a simple transposition such that $\inv(\tau s) = \inv(\tau)-1$. Then $R_{\sigma, \tau} = R_{\sigma, (\tau s) \star s}$, so by Lemma \ref{lem:goodMoves} the morphism $R_{\sigma \star s, \tau s} \to S$ has the same image $S$ as $\Rst \to S$, and is a cohomological equivalence if and only if $\Rst \to S$ is. By induction, this image is $D_{(\tau s) \star (\sigma \star s)^{-1}}(\cpb; \cqb)$ and both maps are cohomological equivalences. Finally, Fact \ref{fact:demazure} implies that $(\tau s) \star (\sigma \star s)^{-1} = (\tau s) \star s^{-1} \star \sigma^{-1} = \tau \star \sigma^{-1}$, so the image of both maps is $D_{\tau \star \sigma^{-1}}(\cpb; \cqb)$. This completes the induction.
\end{proof}

\begin{proof}[Proof of Theorem~\ref{thm:relativeRichardson}]
Part (2) of Theorem \ref{thm:relativeRichardson} is Theorem \ref{thm:rabgood}. Part (1) is Theorem \ref{thm:RabLocal}.
\end{proof}


\section{Brill-Noether varieties and relative Richardson varieties}
\label{sec:brillnoether}

This section describes an example of relative Richardson varieties arising in Brill-Noether theory, which is a crucial ingredient in \cite{chan-pflueger-euler}. Let $E$ be an elliptic curve and $L\in \Pic^d (E)$, and let $V=H^0(E,L)$.  Suppose $p$ and $q$ are distinct closed points on $E$ such that $p-q$ is nontorsion in the Jacobian.  Assume $d \geq 1$. Define two complete flags of $V$ as follows.

\begin{eqnarray*}
P^i = \begin{cases}
V(-i p) & \mbox{ if $0 \leq i < d$}\\
0 & \mbox{ if $i = d$}.
\end{cases} &&
Q^j = \begin{cases}
V(-j q) & \mbox{ if $0 \leq j < d$}\\
0 & \mbox{ if $j = d$}.
\end{cases}
\end{eqnarray*}

Let $\cH$ be the rank $d$ vector bundle on $\Pic^d E$ whose fiber over $L$ is identified with $H^0(E,L)$; more precisely $\cH$ is the pushforward to $\Pic^d E$ of the Poincar\'e bundle on $\Pic^d E \times E.$
For discussion of Poincar\'e line bundles and this construction, see \cite[\S IV.2-3]{acgh}. One must be careful in this construction to assume $d \geq 2g-1$, so that $\cH$ is a vector bundle, and to only define $P^i$ for $d-i \geq 2g-1$ for the same reason. In this case $2g-1 =1$, so we have assumed $d \geq 1$ and $i < d$.
Then the flags $P^\bu$ and $Q^\bu$, defined above for each $L\in \Pic^d E$, globalize to a pair of flags $\cP^\bu, \cQ^\bu$.
\lemnow{\label{lem:g-a-special-case} The flags $\cP^\bu$ and $ \cQ^\bu$ are versal.
}

\begin{proof}
Notice that $\cpb,\cqb$ are transverse except when $L\cong \mathcal{O}(ap+bq)$ for some $a,b> 0$ and $a+b=d$. Since we assume $p-q$ is non-torsion, the integers $a,b$ are unique in this case, and we have

$$\dim P^i \cap Q^j
= \begin{cases}
\max(d-i-j, 0) & \mbox{ if } (i,j) \neq (a,b)\\
1 & \mbox{ if } (i,j) = (a,b).
\end{cases}
$$

It follows that if $L\cong \mathcal{O}(ap+bq)$ then the flags $P^\sbu, Q^\sbu$ are almost-transverse, i.e. their associated permutation is a simple transposition.
By Lemma~\ref{lem:versalityCriterion2}, it is enough to show that the subscheme of $\Pic^d(E)$ over which $\cpb$ and $\cqb$ fail to be transverse consists of exactly these finitely many reduced points; the only issue to check here is reducedness.

Locally around a point $L = \cO_E( a p + b q)$, where $a,b>0$ with $a+b=d$, the scheme is question is defined by the condition $\dim P^a \cap Q^{b} \ge 1$ (where the scheme structure can be defined with a degeneracy condition for a map of vector bundles). This can be reformulated as the condition that $P^{a-1} \cap Q^{b}$ and $P^{a} \cap Q^{b-1}$ are equal, viewing both of these as codimension $1$ subspaces of $P^{a-1} \cap Q^{b-1}$. The assertion of reducedness amounts to showing that two sections of the $\PP^1$-bundle $\PP (\cP^{a-1} \cap \cQ^{b-1}) \rightarrow \Pic^d E$ are transverse. This $\PP^1$-bundle can be identified with $\mathrm{Sym}^2 E$, regarded as a $\PP^1$-bundle over $\Pic^2 E$, by tensoring with $\cO((a-1)p+(b-1)q).$
The two sections are the effective divisors in $\mathrm{Sym}^2 E$ containing $p$ and containing $q$, respectively. These loci meet transversely at $p+q \in \mathrm{Sym}^2 E$, since the tangent space there may be identified with $T_p E \times T_q E$, and the tangent spaces to the curves identified with the tangent spaces of the two factors.
\end{proof}

It follows that the twice-pointed Brill Noether varieties $\Grdab(E,p,q)$, studied in \cite{chan-osserman-pflueger-gieseker, chan-pflueger-euler}, are examples of relative Richardson varieties.  See those papers for the definitions.
\begin{corollary}
The schemes $\Grdab(E,p,q)$ are relative Richardson varieties over $\Pic^d(E)$.
\end{corollary}

\subsection{A conjecture in higher genus} We conclude with a conjectural generalization to higher genus that would provide an intriguing generalization of some results in Brill-Noether theory. To state our conjecture, we generalize the notion of versality to partial flags in the natural way: $\ell$ partial flags on a scheme $S$ are versal if the induced map from $\Fr(\cH)$ to a product of \emph{partial} flag varieties is smooth. One may then define relative Richardson varieties with respect to partial flags. Note that one must impose restrictions on nests of sets $\ab,\bb$ that may be used to define such relative Richardson varieties: $\Ess(\sigma(\ab))$ must contain only elements $(a,b)$ for which $\cP^b$ is defined, and likewise for $\bb$ and $\cqb$.

Let $C$ be a curve of genus $g$ with two marked points $p,q$. Fix an integer $N \geq 2g-1$. For every point $[L] \in \Pic^N(C)$, the vector space $H = H^0(C,L)$ has dimension $N-g+1$ and has two partial flags given by vanishing orders at $p$ and $q$, namely
\begin{eqnarray*}
P^a &=& H^0(C,L(-ap)) \mbox{ for $0 \leq a \leq N-2g+1$ }\\
Q^b &=& H^0(C,L(-bq)) \mbox{ for $0 \leq b \leq N-2g+1$. }
\end{eqnarray*}
The upper bounds on $a,b$ ensure that $P^a$ has codimension $a$ in $H$, since $L(-ap)$ and $L(-bq)$ are nonspecial; this is analogous to the need in \cite[\S IV.3]{acgh} to twist by a fixed divisor in order to work with lines bundles of degree at least $2g-1$. This construction globalizes, giving a vector bundle $\cH$ with partial flags $\cpb,\cqb$, each with coranks $0,1,\cdots,N-2g+1$.

\begin{conj} \label{conj:higherG}
If $(C,p,q)$ is a general twice-marked curve of genus $g$, then for all $N\ge 2g-1$ the flags $\cpb, \cqb$ in $\cH$ described above are versal on $\Pic^N(C)$.
\end{conj}

This conjecture would show that for general curves $C$, Brill-Noether varieties $G^r_d(C)$ are relative Richardson varieties (in the more general sense where partial flags are allowed). Indeed, by choosing $N$ sufficiently large that there are integers $a,b$ with $N-a,N-b \geq 2g-1$ and $N-a-b = d$, one may embed $G^r_d(C)$ in the Grassmannian bundle $\Gr(r+1, \cH)$ by twisting by the divisor $ap + bq$. This gives the following isomorphism.
$$G^r_d(C) \cong \{ (L,V) \in \Gr(r+1, \cH):\ V \subseteq (\cP^a)_x  \mbox{ and } V \subseteq (\cQ^b)_x \},$$
where $\cH$ is the vector bundle on $\Pic^N(C)$ given by $\cH|_{[L]} \cong H^0(C,L(ap+bq)).$
Conjecture \ref{conj:higherG} would therefore imply that $G^r_d(C)$ is isomorphic over $\Pic^d(C)$ to a relative Richardson variety.

More generally, Brill-Noether varieties with ramification $G^{r,\alpha,\beta}_d(C,p,q)$ (see \cite{chan-osserman-pflueger-gieseker} for definitions) may be described in a similar manner, and Conjecture \ref{conj:higherG} also implies that they are isomorphic over $\Pic^d(C)$ to relative Richardson varieties.

We remark that Theorem \ref{thm:kwy} generalizes readily to versal partial flags: one may either replace complete flag varieties with partial flag varieties throughout the proof, or deduce the general result from Theorem \ref{thm:kwy} by locally extending the versal partial flags to versal complete flags. Conjecture \ref{conj:higherG} would give a new way to study singularities of $G^{r,\alpha,\beta}_d(C,p,q)$ for general $(C,p,q)$, and in particular would provide a new proof of the Gieseker-Petri Theorem and the main theorem of \cite{chan-osserman-pflueger-gieseker} characterizing the singular locus of $G^{r,\alpha,\beta}_d(C,p,q)$. It would also generalize these results from linear series to flags of linear series.

\bigskip

\bigskip

%
%

\bibliographystyle{amsalpha}
\bibliography{my}

\providecommand{\bysame}{\leavevmode\hbox to3em{\hrulefill}\thinspace}
\providecommand{\MR}{\relax\ifhmode\unskip\space\fi MR }
\providecommand{\MRhref}[2]{%
  \href{http://www.ams.org/mathscinet-getitem?mr=#1}{#2}
}
\providecommand{\href}[2]{#2}
\begin{thebibliography}{ACGH85}

\bibitem[ACGH85]{acgh}
E.~Arbarello, M.~Cornalba, P.~A. Griffiths, and J.~Harris, \emph{Geometry of
  algebraic curves. {V}ol. {I}}, Grundlehren der Mathematischen Wissenschaften,
  vol. 267, Springer-Verlag, New York, 1985.

\bibitem[ACT22]{anderson-chen-tarasca-k-classes}
Dave Anderson, Linda Chen, and Nicola Tarasca, \emph{{$K$}-classes of
  {B}rill-{N}oether loci and a determinantal formula}, Int. Math. Res. Not.
  IMRN (2022), no.~16, 12653--12698. \MR{4466009}

\bibitem[BGG73]{bernstein-gelfand-gelfand-schubert}
I.~N. Bern\v{s}te\u{\i}n, I.~M. Gel{\cprime}fand, and S.~I. Gel{\cprime}fand,
  \emph{Schubert cells, and the cohomology of the spaces {$G/P$}}, Uspehi Mat.
  Nauk \textbf{28} (1973), no.~3(171), 3--26. \MR{0429933}

\bibitem[Bri05]{brion-lectures}
Michel Brion, \emph{Lectures on the geometry of flag varieties}, Topics in
  cohomological studies of algebraic varieties, Trends Math., Birkh\"auser,
  Basel, 2005, pp.~33--85. \MR{2143072}

\bibitem[BW03]{billey-warrington-maximal}
Sara~C. Billey and Gregory~S. Warrington, \emph{Maximal singular loci of
  {S}chubert varieties in {${\rm SL}(n)/B$}}, Trans. Amer. Math. Soc.
  \textbf{355} (2003), no.~10, 3915--3945. \MR{1990570}

\bibitem[COP19]{chan-osserman-pflueger-gieseker}
Melody Chan, Brian Osserman, and Nathan Pflueger, \emph{The {G}ieseker--{P}etri
  theorem and imposed ramification}, Bulletin of the London Mathematical
  Society \textbf{51} (2019), no.~6, 945--960.

\bibitem[Cor01]{cortez-singularities}
Aur\'{e}lie Cortez, \emph{Singularit\'{e}s g\'{e}n\'{e}riques et
  quasi-r\'{e}solutions des vari\'{e}t\'{e}s de {S}chubert pour le groupe
  lin\'{e}aire}, C. R. Acad. Sci. Paris S\'{e}r. I Math. \textbf{333} (2001),
  no.~6, 561--566. \MR{1860930}

\bibitem[CP21]{chan-pflueger-euler}
Melody Chan and Nathan Pflueger, \emph{Euler characteristics of
  {B}rill-{N}oether varieties}, Trans. Amer. Math. Soc. \textbf{374} (2021),
  no.~3, 1513--1533.

\bibitem[Dem74]{demazure-desingularisation}
Michel Demazure, \emph{D\'{e}singularisation des vari\'{e}t\'{e}s de {S}chubert
  g\'{e}n\'{e}ralis\'{e}es}, Ann. Sci. \'{E}cole Norm. Sup. (4) \textbf{7}
  (1974), 53--88. \MR{354697}

\bibitem[FP98]{fulton-pragacz}
William Fulton and Piotr Pragacz, \emph{Schubert varieties and degeneracy
  loci}, Lecture Notes in Mathematics, vol. 1689, Springer-Verlag, Berlin,
  1998, Appendix J by the authors in collaboration with I. Ciocan-Fontanine.
  \MR{1639468}

\bibitem[Ful92]{fulton-flags}
William Fulton, \emph{Flags, {S}chubert polynomials, degeneracy loci, and
  determinantal formulas}, Duke Math. J. \textbf{65} (1992), no.~3, 381--420.
  \MR{1154177}

\bibitem[Ful97]{fulton-young}
\bysame, \emph{Young tableaux}, London Mathematical Society Student Texts,
  vol.~35, Cambridge University Press, Cambridge, 1997, With applications to
  representation theory and geometry. \MR{1464693}

\bibitem[Har77]{hartshorne-algebraic}
Robin Hartshorne, \emph{Algebraic geometry}, Springer-Verlag, New
  York-Heidelberg, 1977, Graduate Texts in Mathematics, No. 52. \MR{0463157}

\bibitem[KLR03]{kassel-lascoux-reutenauer-singular}
Christian Kassel, Alain Lascoux, and Christophe Reutenauer, \emph{The singular
  locus of a {S}chubert variety}, J. Algebra \textbf{269} (2003), no.~1,
  74--108. \MR{2015302}

\bibitem[Kov17]{kovacs-rational}
Sándor~J Kovács, \emph{Rational singularities}, arXiv:1703.02269, 2017.

\bibitem[KWY13]{knutson-woo-yong-singularities}
Allen Knutson, Alexander Woo, and Alexander Yong, \emph{Singularities of
  {R}ichardson varieties}, Math. Res. Lett. \textbf{20} (2013), no.~2,
  391--400. \MR{3151655}

\bibitem[LS90]{lakshmibai-sandhya-criterion}
V.~Lakshmibai and B.~Sandhya, \emph{Criterion for smoothness of {S}chubert
  varieties in {${\rm Sl}(n)/B$}}, Proc. Indian Acad. Sci. Math. Sci.
  \textbf{100} (1990), no.~1, 45--52. \MR{1051089}

\bibitem[Man01]{manivel-lieu}
L.~Manivel, \emph{Le lieu singulier des vari\'{e}t\'{e}s de {S}chubert},
  Internat. Math. Res. Notices (2001), no.~16, 849--871. \MR{1853139}

\bibitem[Per09]{perrin-gorenstein}
Nicolas Perrin, \emph{The {G}orenstein locus of minuscule {S}chubert
  varieties}, Adv. Math. \textbf{220} (2009), no.~2, 505--522. \MR{2466424}

\bibitem[{Sta}17]{stacks-project}
The {Stacks Project Authors}, \emph{{Stacks Project}},
  \url{http://stacks.math.columbia.edu}, 2017.

\bibitem[WY06]{woo-yong-when}
Alexander Woo and Alexander Yong, \emph{When is a {S}chubert variety
  {G}orenstein?}, Adv. Math. \textbf{207} (2006), no.~1, 205--220. \MR{2264071}

\end{thebibliography}

\end{document}